\documentclass[12pt,letterpaper]{amsart}

\usepackage[centertags]{amsmath}
\usepackage{amsthm}
\usepackage{amssymb}
\usepackage{amscd}
\usepackage{eucal}
\usepackage{epsfig}
\usepackage{verbatim}
\usepackage{color}
\usepackage{bbm}
\usepackage{latexsym,enumerate,amsxtra}
\usepackage[all]{xy}

\newtheorem{theoremA}{Theorem}

\newtheorem{thm}{Theorem}[subsection]
\newtheorem{theorem}[thm]{Theorem}
\newtheorem{corollary}[thm]{Corollary}
\newtheorem{lemma}[thm]{Lemma}
\newtheorem{proposition}[thm]{Proposition}
\newtheorem{definition}[thm]{Definition}

\newtheorem{remarks}[thm]{Remarks}
\newtheorem{remark}[thm]{Remark}

\newtheorem{numbering}[thm]{}

\newtheorem{hypo}{Hypothesis}

\newcommand{\CaA}{\mathcal A}
\newcommand{\CaB}{\mathcal B}

\newcommand{\CaF}{\mathcal F}

\newcommand{\CaH}{\mathcal H}

\newcommand{\CaR}{\mathcal R}
\newcommand{\CaS}{\mathcal S}

\newcommand{\CaU}{\mathcal U}

\newcommand{\bbF}{\mathbb F}

\newcommand{\bbR}{\mathbb R}
\newcommand{\bbZ}{\mathbb Z}
\newcommand{\tbR}{\widetilde{\mathbb R}}
\newcommand{\chf}{\mathbbm{1}}

\newcommand{\sfX}{\mathsf X}

\newcommand{\val}{\nu}
\newcommand{\Ad}{\textup{Ad}}

\newcommand{\sms}{{ss}}
\newcommand{\as}{{\mathrm{a}}}
\newcommand{\tu}{{0^+}}

\newcommand{\bG}{\mathbb G}
\newcommand{\bU}{\mathbb U}
\newcommand{\bH}{\mathbb H}

\newcommand{\bC}{\mathbb C}

\newcommand{\bT}{\mathbb T}
\newcommand{\bM}{\mathbb M}
\newcommand{\bP}{\mathbb P}
\newcommand{\bN}{\mathbb N}
\newcommand{\bS}{\mathbb S}

\newcommand{\bX}{\sfX}
\newcommand{\stab}{\mathrm{Stab}}
\newcommand{\T}{T}

\newcommand{\Z}{Z}
\newcommand{\C}{C}
\newcommand{\G}{G}

\newcommand{\lieG}{\mathfrak g}
\newcommand{\lieM}{\mathfrak m}
\newcommand{\lieN}{\mathfrak n}
\newcommand{\lieT}{\mathfrak t}

\newcommand{\blieG}{\boldsymbol{\lieG}}

\newcommand{\dpl}{\mathsf d}
\newcommand{\dist}{\mathrm{dist}}

\newcommand{\Set}{S}

\newcommand{\bilinear}{\mathrm B}
\newcommand{\Bd}{\CaB}

\newcommand{\Apt}{\CaA}

\newcommand{\chamber}{C}
\newcommand{\chamberD}{D}

\newcommand{\Facet}{\CaF}
\newcommand{\oface}{\overline F^\ast}
\newcommand{\face}{F^\ast}

\newcommand{\dth}{d}

\newcommand{\rtm}{r}
\newcommand{\rg}{{\mathrm g}}
\newcommand{\atm}{a}
\newcommand{\btm}{b}

\newcommand{\supp}{\textup{Supp}}

\newcommand{\nil}{\mathcal N}

\newcommand{\Hypk}{\textup{(H$k$)}}
\newcommand{\HypB}{\textup{(HB)}}
\newcommand{\HypGT}{\textup{(HGT)}}
\newcommand{\HypN}{\textup{(H$\nil$)}}

\newcommand{\midvsp}{\vspace{7pt}}

\def\rig{\mathrm{rig}}
\newcommand{\BZ}{\mathbb {Z}}
\def\BR{\mathbb R}
\def\BG{\mathbb G}
\def\BN{\mathbb N}
\def\CO{\mathcal O}
\def\CH{\mathcal H}
\def\g{\gamma}
\def\a{\alpha}
\def\e{\epsilon}
\def\CS{\mathcal S}

\newcommand{\remind}[1]{{\bf ** #1 **}}

\begin{document}

\title{Jordan Decompositions of cocenters of reductive $p$-adic groups} 
 \author{Xuhua He and Ju-Lee Kim}
\address{Department of Mathematics, University of Maryland, College Park, MD 20742, USA} \email{xuhuahe@math.umd.edu}
\address{Department of Mathematics, Massachusetts Institute of Technology, 77 Massachusetts Avenue, Cambridge MA 02139, USA}
\email{juleekim@mit.edu}
\thanks{X. H. was partially supported by NSF DMS-1463852 and DMS-1128155 (from IAS)}

\begin{abstract}
Cocenters of Hecke algebras $\CaH$ play an important role in studying mod $\ell$ or $\mathbb C$ harmonic analysis on connected $p$-adic reductive groups. On the other hand, the depth $r$ Hecke algebra $\CaH_{r^+}$ is well suited to study depth $r$ smooth representations. In this paper, we study depth $r$ rigid cocenters $\overline\CaH^\rig_{r^+}$ of a connected reductive $p$-adic group over rings of characteristic zero or $\ell\neq p$. More precisely, under some mild hypotheses, we establish a Jordan decomposition of the depth $r$ rigid cocenter, hence find an explicit basis of $\overline\CaH^\rig_{r^+}$. 
\end{abstract}

\maketitle

\section*{\bf Introduction}

\subsection{} Let $G$ be a connected reductive $p$-adic group. Let $R$ be an algebraically closed field of characteristic not equal to $p$. Let $\CH_R$ be the Hecke algebra of locally constant, compactly supported $R$-valued functions on $G$. The trace map $$Tr_R: \overline \CH_R \to \mathfrak R_R(G)^*$$ relates the cocenter $\overline \CH_R=\CH_R/[\CH_R, \CH_R]$ and the Grothendieck group $\mathfrak R_R(G)$ of smooth admissible representations of $G$ over $R$. 

In most cases, the cocenter is expected to be ``dual'' to the representations. For $R=\mathbb C$, Bernstein, Deligne and Kazhdan in \cite{BDK} and \cite{Kaz} proved the trace map $Tr_{\mathbb C}: \overline \CH_{\mathbb C} \xrightarrow{\cong} \mathfrak R_{\mathbb C}(G)^*_{good}$ is a bijection between  the cocenter and the ``good linear forms'' on $\mathfrak R_{\mathbb C}(G)$. For modular representations over $R$, the surjection $Tr_{R}: \overline \CH_{R} \to \mathfrak R_{R}(G)^*_{good}$ is established in \cite{hecke-3} under the assumption that the cardinality of the relative Weyl group of $G$ is invertible in $R$. It is conjectured that the injection holds if the pro-$p$ order of any open compact subgroup of $G$ is invertible in $R$.

This motivates our study of the structure of the cocenter of the Hecke algebra. To be precise, we mainly consider the integral form $\CH=\CH(G)$, i.e. $\CaH_R(G)$ with $R=\mathbb Z[\frac 1p]$. This will allow us to apply the results on $\overline \CH$ to both the ordinary and the modular representations of $G$. 

\subsection{}\label{0.2} In \cite[Theorem B]{hecke-1} and \cite[Theorem C \& Theorem 6.5]{hecke-2}, the first named author showed that $$\overline \CH \cong \oplus_M \overline \CH(M)^{\rig, +},$$ where $M$ runs over all the standard Levi subgroups of $G$ and $\overline \CH(M)^{\rig, +}$ is the $+$-rigid part of the cocenter of the Hecke algebra $\CH(M)$, i.e. the $\BZ[\frac{1}{p}]$-submodule of $\overline \CH(M)$ consisting of elements represented by the functions supported in the compact-modulo-center elements of $M$ whose Newton points are dominant (in $G$) and  with centralizer equal to $M$. In other words, the rigid cocenters of the Hecke algebras of various standard Levi subgroups form the ``building block'' of the whole cocenter $\overline \CH$. We refer the details to \emph{loc. cit.}.

In this paper, we study the rigid cocenter $\overline \CH^{\rig}$, the $\BZ[\frac{1}{p}]$-submodule of $\overline \CH$ represented by functions supported in the subset $G^\rig$ of compact-modulo-center elements of $G$. More precisely, we focus on the depth $r$ rigid cocenter $\overline \CH_{r^+}^\rig$ for any real number $r>0$, defined as follows. 

For any element $x$ in the reduced Bruhat--Tits building $\Bd(G)$ of $G$, Moy and Prasad \cite{MP} associated a subgroup $G_{x, r^+}$ of $G$. Let $\CH_{r^+}=\sum_{x \in \Bd(G)} C_c(G/G_{x, r^+})$ and $\overline \CH_{r^+}$ be its image in $\overline \CH$, the \emph{depth $r$ cocenter}. The \emph{depth $r$ rigid cocenter} $\overline \CH^{\rig}_{r^+}=\overline \CH^\rig \cap \overline \CH_{r^+}$. According to Howe's conjecture, this is a finitely generated $\BZ[\frac{1}{p}]$-module. Moreover, we have that 
$\overline \CH^\rig=\varinjlim\limits_r \overline \CH^{\rig}_{r^+}$. 

\subsection{} The main purpose of this paper is to establish the ``Jordan decomposition'' of $\overline \CH^{\rig}_{r^+}$. 

Before stating the main result, we make a short digression and discuss a ``toy model'', the cocenter of the group algebra $\BZ[H]$ of a finite reductive group $H$. 

For any element $g \in H$, we have the Jordan decomposition $g=g_s g_u$, where $g_s$ is the semisimple part of $g$ and $g_u$ is the unipotent part of $g$. Then we have the Jordan decomposition of the group algebra $\BZ[H] \cong \oplus_{s \in H^\sms} \BZ[C_H(s)^{unip}]$, where $H^\sms$ is the set of semisimple elements of $H$, $C_H(s)$ is the centralizer of $s$ and $C_H(s)^{unip}$ is the set of unipotent elements in $C_H(s)$. Based on the Jordan decomposition on the group algebra $\BZ[H]$, one deduces the Jordan decomposition of the cocenter 
$$
\overline{\BZ[H]} 
:=\BZ[H]\left/[\BZ[H],\BZ[H]]\right.
\cong\oplus_{[s] \in Cl^\sms(H)} \overline{\BZ[C_H(s)^{unip}]},
$$ 
where $Cl^\sms(H)$ is the set of semisimple conjugacy classes of $H$ and $\overline{\BZ[C_H(s)^{unip}]}$ is the image of $\BZ[C_H(s)^{unip}]$ in $\overline{\BZ[C_H(s)]}$, which is a free $\BZ$-module with basis indexed by the unipotent conjugacy classes of $C_H(s)$. 

\subsection{} Now we come back to connected reductive $p$-adic groups. As any element in the Hecke algebra $\CH$ is a locally constant function, there seems no analogous Jordan decomposition on $\CH$. However, under the hypotheses in \S\ref{subsec: hypos}, we have the analogous part of semisimple conjugacy classes and unipotent conjugacy classes in the context of the cocenter of Hecke algebras. 

By the work of Adler and Spice \cite{AS}, we may write a semisimple compact-modulo-center element $\g$ as a ``good product''. Since we are working with the cocenter $\overline \CH^\rig_{r^+}$ of depth $r$, we use the truncated part $\g_{\le r}$ of $\g$. 
The equivalence classes $\CS_r$ of semisimple compact-modulo-center elements of $G$, roughly speaking, are generated by the conjugation action and the truncated operation (see \S\ref{sec: ss} for the precise definition). The set $\CS_r$ is the analogue of semisimple conjugacy classes and serves as the index set of the desired Jordan decomposition on $\overline \CH^\rig_{r^+}$. 

For any $[\g] \in \CS_r$, we pick up the truncation $\g_{\le r}$ (see Definition \ref{defn: good product} for details) of a representative $\g \in [\g]$ and denote by $C_G(\g_{\le r})$ the centralizer of $\g_{\le r}$. The isomorphism class of $C_G(\g_{\le r})$ is independent of the choice of $\g$ and its truncation $\g_{\le r}$.

Now we come to the unipotent part. Let $\overline \CH^{G,\flat}_{r^+}$ be the $\BZ[\frac{1}{p}]$-submodule of $\overline \CH^{\rig}_{r^+}$, represented by functions in $\sum_{x \in \Bd(G)} C_c(G_{x, r}/G_{x, r^+})$ with support in $G_{r^+}:=\cup_{x \in \Bd(G)} G_{x, r^+}$. Based on the work of DeBacker in \cite{De1} and \cite{De2}, $\overline \CH^{G,\flat}_{r^+}$ is a free module with basis indexed by the unipotent conjugacy classes of $G$. This is the analogy of the set of unipotent conjugacy classes, or in other words, the analogy of $\overline{\BZ[H^{unip}]}$ in the cocenter of the group algebra $\BZ[H]$. 
Now we state the main results of this paper. 

\begin{theoremA}[Theorem \ref{JD-H} \& Theorem \ref{JD-4}]\label{thmA} Fix $r\in\mathbb R_{>0}$.
Suppose Hypotheses in \S\ref{subsec: hypos} hold. Then
$$
\overline \CH^{\rig}_{r^+} \cong \oplus_{[\g] \in \CS_r} \overline \CH^{C_G(\g_{\le r}),\flat}_{r^+}.
$$ 
Moreover, $\overline \CH^{\rig}_{r^+}$ is a free $\BZ[\frac{1}{p}]$-module.  
\end{theoremA}

\begin{theoremA}[Theorem \ref{thmB'}]\label{thmB} Let $J(G^\rig)$ denote the space of $\mathbb C$-valued invariant distributions of $G$ with support on $G^\rig$. 
Suppose Hypotheses in \S\ref{subsec: hypos} hold. Then the restriction $J(G^\rig) \mid_{\CH_{r^+, \mathbb C}}$ has a basis given by the restriction of orbital integrals $O_{\g_{\le r} u}$ to ${\CH_{r^+, \mathbb C}}$, where $[\g] \in \CS_r$, and $u$ runs over the representatives of the unipotent conjugacy classes of $C_G(\g_{\le r})$. 
\end{theoremA}

Theorem \ref{thmB}, together with the Newton decomposition in \cite{hecke-1}, gives a precise estimate on the Howe's conjecture on the restriction of invariant distributions. For more details, see the discussion in \S\ref{howe-diss}.

\subsection{} In \S\ref{sec: prelim}, we review some background materials on Moy--Prasad filtration subgroups and the cocenter $\CaH$ of $G$. Toward the decomposition of $\overline \CH^{\rig}_{r^+}$ in Theorem \ref{thmA}, in \S2, we first decompose $G^\rig$ into a disjoint union of $G$-domains $X_{[\gamma]}:=\,^G\!\left(\gamma C_G(\gamma_{\le r})_{r^+}\right)$ parameterized by $[\gamma]\in\CaS_r$. We use good products of semisimple elements (\cite{AS}) to prove that $G^\rig=\cup_{[\gamma]\in\CaS_r} X_{[\gamma]}$. The Lie algebra version of such decompositions can be found in \cite[\S7]{Kim}. Then, it is easy to see that there is a corresponding decomposition of $\CaH^\rig$ according to this decomposition: $\CH^{\rig}=\oplus_{[\gamma]}\CaH(X_{[\gamma]})$  (Theorem \ref{domain}) where $\CaH(X_{[\gamma]})$ is the submodule consisting of $f\in\CaH_{r^+}$ with $\supp(f)\subset X_{[\gamma]}$. However, since each domain $X_{[\gamma]}$ is not necessarily $G_{x,r^+}$ bi-invariant, $\oplus_{[\g] \in \CS_r}  \CaH_{r^+}(X_{[\gamma]})$ is in fact a proper submodule of $\CH^\rig_{r^+}$ (see \S\ref{subsec: thmA}). 

Now, Theorem \ref{thmA} asserts that the desired decomposition holds at the level of cocenters. 
In \S\ref{sec: descents}, we prove Theorem \ref{thmA} via the following strategy: we first represent elements in $\overline\CaH^\rig_{r^+}$ by elements in $\CaH_{0,r^+}:=\sum_{x\in\Bd(G)} C_c\left(Z(G)\stab_G(x)/G_{x,r^+}\right)$ in the cocenter, and then represent elements in $\overline\CaH_{0,r^+}$ by elements in $\oplus_{[\gamma]} \overline \CaH(X_{[\gamma]})$.
In these steps, we use the descent arguments developed by Howe, Harish-Chandra, Waldspurger, and most recently by DeBacker.  Especially, DeBacker's arguments in \cite{De2} are aptly adaptable in our situations in view of recent developments in harmonic analysis in $p$-adic groups. As a result, most of our hypotheses are inherited from \cite{AS} and \cite{De2}. Lastly, we prove $\overline\CaH_{r^+}(X_{[\gamma]})\simeq\overline\CH^{C_G(\g_{\le r}),\flat}_{r^+}$ using inductive descents (see Proposition \ref{good-i}). A Lie algebra version of inductive descents can be found in \cite[\S6]{KMu2}. 

In \S\ref{sec: JD}, we prove Theorem \ref{thmB}. We combine inductive descents and the parameterization of unipotent conjugacy classes in \cite{De2}. However, since the centralizer of a semisimple element is not necessarily connected in this paper, one needs to adapt DeBacker's parameterization for our case.  

In \S\ref{sec: ex}, we present examples to illustrate the duality between cocenters and representations.

\

\noindent
{\it Acknowledgments.} We would like to thank Anne-Marie Aubert for her valuable comments. The second author would like to thank Roger Howe for his question which lead the authors to add \S5 of this paper.
This work has been done while both authors were members at the Institute for Advanced Study during 2016-1017. We would like to thank their hospitality and support.

\subsection*{\bf Notation and Conventions}\label{sec: notation}

Let $F$ be a locally compact field 
with finite residue field $\bbF_{p^n}$.
Let $\bG$ be a connected reductive group defined over $F$. For any finite extension $E$ of $F$, 
let $\bG(E)$ be the group of $E$-rational points of $\bG$. We will simply write $\G$ for $\bG(F)$.
Denote the Lie algebras of $\bG$ and $\bG(E)$ by $\blieG$ and
$\blieG(E)$, respectively. 
In general, we use bold characters $\bH,\,\bM$ and $\bN$, etc
to denote algebraic groups. If they are defined over $F$, we will use
corresponding Roman characters $H,\,M$ and $N$ to denote the groups of
$F$-points, and $\mathfrak h,\,\lieM$ and $\lieN$ to denote
the Lie algebras of $H,\,M$ and $N$. Let $\overline \bG:=(\bG\left/\Z(\bG)\right.)^\circ$ where $\Z(\bG)$ is the center of $\bG$.

We denote by $G^\sms$ the set of semisimple elements in $G$, by $\CaU$ the set of unipotent elements in $G$, and by $G^\rig$ the set of compact-modulo-center elements in $G$. 

We let $\mu_G$ denote a fixed Haar measure on $G$.



For $g\in G$, ${ }^g \!X$ denotes $gXg^{-1}$ and
for $S, H\subset\G$, $^H\!S:=\{\,{}^g\!X \mid X\in S,\ g\in H\}$. 

We set $\tilde \BR=\BR \sqcup \{r^+; r \in \BR\} \sqcup \{\infty\}$ and define the partial order on $\tilde \BR$ as follows: 
for $r, s \in \BR$, $r<s^+$ if $r \le s$, $r^+<s^+$ and $r^+<s$ if $r<s$, and $r, r^+<\infty$ for any $r \in \BR$. 

We denote by $\CH$, the Hecke algebra of locally constant, compactly supported $\BZ[\frac{1}{p}]$-valued functions on $G$. The cocenter $\overline \CH=\CH/[\CH, \CH]$. Let $\CH^\rig$ be the $\BZ[\frac{1}{p}]$-submodule of $\CH$ consisting of functions supported in $G^\rig$. The rigid cocenter $\overline \CH^\rig$ is the image of $\CH^\rig$ in $\overline \CH$.

\midvsp

\ 

\section {\bf Preliminaries}\label{sec: prelim}

\subsection{Moy-Prasad Filtrations}\label{subsec: MP filtrations}

\begin{numbering}{\bf Apartments and buildings.}\rm \ 
For a finite extension $E$ of $F$, 
let $\Bd(\bG, E)$ denote the extended Bruhat-Tits building of $\bG$ over $E$. 
Recall that 
$\Bd(\bG, E)
\simeq\Bd(\overline{\bG},E)\times\left(\bX_\ast(Z(\bG),E)\otimes\bbR\right)$, and
$\bX_\ast(\Z(\bG),E)$ is the abelian group 
of $E$-rational cocharacters of
the center $Z(\mathbb G)$ of $\bG$.
If $\bT$ is a maximal $F$-torus in $\bG$ which
splits over $E$, let $\Apt(\bT,E)$ be the corresponding apartment 
over $E$. It is known that for any tamely ramified
finite Galois extension $E'$ of $E$, 
$\Bd(\bG,E)$ can be embedded into $\Bd(\bG,E')$
and its image is equal to the set of the Galois fixed points in $\Bd(\bG,E')$
(see \cite[(5.11)]{Rous} or \cite{Pr}).

For a maximal $F$-torus $\bT$ in $\bG$ which splits over a tamely ramified finite Galois extension $E$ of $F$, 
we write $\Apt(\bT,F)$ for $\Apt(\bT,E)\cap\Bd(\bG,F)$. This is
well defined independent of the choice of $E$. Moreover,
$\Apt(\bT,F)$ is the set of Galois fixed points in $\Apt(\bT,E)$.
For simplicity, we write $\Bd(\G)=\Bd(\bG,F)$, $\Apt(\T)=\Apt(\bT,F)$ etc. 
\end{numbering}

\begin{numbering}{\bf  Moy-Prasad filtrations.}\rm \ 
Regarding $\bG$ as a group defined over $E$, Moy and Prasad associate 
$\blieG(E)_{x,r}$ and $\bG(E)_{x,|r|}$ (resp. $\blieG(E)_{x,r^+}$ and $\bG(E)_{x,r^+}$) to $(x,r)\in\Bd(\bG,E)\times\bbR$ with respect to the valuation normalized as follows \cite{MP2}:
let $E^u$ be the maximal unramified extension of $E$,
and $E$ the minimal extension of $E^u$ over which $\bG$ splits.
Then the valuation used by Moy and Prasad maps $L^\times$ onto $\bbZ$.

In this paper, we let $\val=\val_F$ be the valuation on $F$ such that $\val(F^\times)=\bbZ$, $\val_E$ extends $\val$. Let $\overline F$ be an algebraic closure of $F$.
For an extension field $E$ of $F$, let $\val_E$ be the valuation
on $E$ extending $\val$. We will just write $\val$ for $\val_E$.
Then, with respect to our normalized valuation $\val$,
we can define filtrations in $\blieG(E)$ and $\bG(E)$.
Then our $\blieG(E)_{x,r}$ and $\bG(E)_{x,r}$ correspond to 
$\blieG(E)_{x,elr}$ and $\bG(E)_{x,elr}$ of Moy and Prasad,
where $e=e(E/F)$ is the ramification index of $E$ over $F$ and
$l=[L:E^u]$. 

This normalization is chosen to have the following property 
\cite[(1.4.1)]{Ad}: 

\begin{enumerate}
\item For a tamely ramified Galois extension $E'$ of $E$ 
and $x\in\Bd(\bG,E)\subset\Bd(\bG,E')$, for $r\in\widetilde{\mathbb R}$, we have
\[
\blieG(E)_{x,r}=\blieG(E')_{x,r}\cap\blieG(E).
\]
If $r>0$, 
\[
\bG(E)_{x,r}=\bG(E')_{x,r}\cap\bG(E).
\]
\item For $\rtm\in\frac1e\bbZ$, two points $x$ and $y$ in $\Bd(\bG,E)$ lie in
the same facet if and only if
\[
\bG(E)_{x,\rtm}=\bG(E)_{y,\rtm} \quad\textrm{and}\quad
\bG(E)_{x,\rtm^+}=\bG(E)_{y,\rtm^+}\,.
\]
\end{enumerate}
\end{numbering}

\smallskip

\begin{numbering} \rm
For simplicity, we put $G_{x,\rtm}:=\bG(F)_{x,\rtm}$, etc.
We will also use the following notation. For $\rtm\in\bbR_{\ge0}$, let
\[
\G_r=\cup_{x\in\Bd(\G)} \G_{x,r},\qquad \G_{r^+}=\cup_{s>r} \G_s.
\]

Let $\Phi(\bT,\bG,E)$ be the set of $E$-roots of $\bT$ in $\bG$,
and let $\Psi(\bT,\bG,E)$ be the corresponding set of affine roots in $\bG$.
If $\psi\in\Psi(\bT,\bG,E)$, 
let $\dot\psi\in\Phi(\bT,\bG,E)$ be the gradient of $\psi$, 
and let $\bU(E)_{\dot\psi}\subset\bG(E)$ be the root group 
corresponding to $\dot\psi$.
We denote the root subgroup in $\bU(E)_{\dot\psi}$ corresponding
to $\psi$ by $\bU(E)_\psi$.

Let $\bX_\ast(\bT,E)$ be the set of cocharacters of $\bT$, 
and let $\bX^\ast(\bT,E)$ be the set of characters of $\bT$. Let $T_0$ be the maximal compact subgroup of $T$.
For $r\ge0$, set
\begin{align*}
T_r:&=\{t\in T_0\mid\nu(\chi(t)-1)\ge r\textrm{ for all }\chi\in\bX^\ast(\bT, E)\}, \\
Z_r:&=T_r\cap Z_G.
\end{align*}
Note that $Z_r$ is well defined independent of the choice of $T$.

In the rest of this paper, $E$ will denote a tamely ramified finite extension of $F$ unless otherwise stated.

\end{numbering}

\smallskip

\subsection{Cocenters}  \ 

\smallskip

\begin{numbering} \rm
For $s \in \tilde \BR_{\ge 0}$, let $\CH(G, G_{x, s})$ be the space of compactly supported, $G_{x, s} \times G_{x, s}$-invariant $\BZ[\frac{1}{p}]$-valued functions on $G$ and $C_c(G/G_{x, s})$ be the space of compactly supported, right $G_{x, s}$-invariant $\BZ[\frac{1}{p}]$-valued functions on $G$. Note that for any $g \in G$ and $x \in \Bd(G)$, we have 
$$\mathbbm{1}_{g G_{x, s}} \equiv \frac{\mu_G(G_{x, s})}{\mu_G(G_{x, s} g G_{x, s})} \mathbbm{1}_{G_{x, s} g G_{x, s}} \mod [\CH, \CH].$$ Thus $\CH(G, G_{x, s})$ and $C_c(G/G_{x, s})$ have the same image in $\overline \CH$. We denote by $\overline \CH_s$ the image of $\CH_{s}=\sum_{x \in \Bd(G)} C_c(G/G_{x, s})$ in $\bar\CH$. Then $\overline \CH=\varinjlim\limits_s\overline \CH_{s}$. 

We set $\overline \CH^{\rig}_{s}=\overline \CH^{\rig} \cap \overline \CH_{s}$. Then $\overline \CH^\rig=\varinjlim\limits_s\overline \CH^\rig_{s}$. 
\end{numbering}

A $G$-domain, by definition, is an open and closed subset of $G$ that is stable under the conjugation action of $G$. We have the following simple facts about on the cocenter $\overline \CH$. 

\begin{lemma}\label{f-g-f}
Let $X$ be a $G$-domain and $\CH(X)$ be the $\BZ[\frac{1}{p}]$-submodule of $\CaH$ consisting of functions supported in $X$. Then $\CH(X) \cap [\CH, \CH]$ is spanned by $f-{}^g\! f$ for $f \in \CH(X)$ and $g \in G$. 
\end{lemma}

\begin{proof}
Let $f \in \CH(X) \cap [\CH, \CH]$. By \cite[Proposition 1.1]{hecke-1}, $f=\sum_i (f_i-{}^{g_i} f_i)$, where $f_i \in \CH$ and $g_i \in G$. Since $X$ is a $G$-domain, $f_i \mid_X \in \CH(X)$ and $({}^{g_i} f_i) \mid_X={}^{g_i} (f_i \mid_X)$ for any $i$. Thus $f=\sum_i (f_i \mid_X-{}^{g_i} f_i \mid_X)$. 
\end{proof}

\begin{lemma}\label{domain}
Let $\{X_\a\}_{\alpha\in I}$ be a family of $G$-domains in $G$ such that $X_\a \cap X_{\a'}=\emptyset$ for any $\a \neq \a'$. Then $\sum_{\a\in I} \overline \CH(X_\a) \subset \overline \CH$ is a direct sum. Here $\overline \CH(X_\a)$ is the image of $\CH(X_\a)$ in $\overline \CH$. 

If moreover $G=\sqcup_{\a\in I} X_{\a}$, then $\overline \CH=\oplus _{\a}\overline \CH(X_\a)$. 
\end{lemma}

\begin{proof}
Let $f_\a \in \CH_{X_\a}$, $\alpha\in I$ such that $\Gamma:=\{\a; f_\a \neq 0\}$ is a finite set. Suppose that $\sum_{\a \in \Gamma} f_\a \in [\CH, \CH]$. Then by \cite[Proposition 1.1]{hecke-1}, there exists finitely many pairs $(f_i, x_i) \in \CH \times G$ such that 
\[\tag{a} \sum_{\a \in \Gamma} f_\a=\sum_i (f_i-{}^{x_i} f_i).\] 
Restricting both sides of (a) to $X_\a$, we have $f_\a=\sum_i (f_i \!\mid_{X_\a}-({}^{x_i} f_i) \!\mid_{X_\a})$.  Since $X_\a$ is a $G$-domain, we have $f_i \!\mid_{X_\a} \in \CH$ and $({}^{x_i} f_i) \!\mid_{X_\a}={}^{x_i} (f_i \!\mid_{X_\a})$. 

Thus $f_\a=\sum_i (f_i \!\mid_{X_\a}-{}^{x_i} (f_i \!\mid_{X_\a})) \in [\CH, \CH]$. Therefore the image of $f_\a$ in $\overline \CH$ is zero and $\sum_\a \overline \CH(X_\a) \subset \overline \CH$ is a direct sum.

If moreover $G=\sqcup_\a X_\a$, then for any $f \in \CH$, $f=\sum_{\a} f \!\mid_{X_\a} \in \sum_{\a} \CH(X_\a)$. Hence $\overline \CH=\sum \overline \CH(X_\a)$. By what we proved above, this is a direct sum. 
\end{proof}

\section{\bf Semisimple Elements and Decomposition of $G^\rig$}\label{sec: ss}

From now on, let $r$ be a positive real number. 

\subsection{\bf Depth functions and good elements}\label{subsec: defns}

If $G$ is semisimple, the following definitions in \ref{defn: depth} and \ref{defn: good} coincide with those in \cite{AS}.

\begin{definition}\label{defn: depth}
\rm Write $Z:=Z(G)$.
For $x\in\Bd(G)$, define the depth-mod-center function 
\[
\dth^{\bG}(x,\ ):\Z\, \stab_G(x)\rightarrow\mathbb R \sqcup \{\infty\},
\]
such that 
\[
\dth^{\bG}(x,g)=\begin{cases} \qquad 0&\textrm{if } g\in \Z\stab_G(x)\setminus \Z G_{x,0^+}\\
\max\{s\mid zg\in G_{x,s} \textrm{ for some }z\in Z\} &\textrm{if }g\in \Z G_{x,0^+}\setminus Z\\
\qquad\infty&\textrm{if }g\in Z.
\end{cases}
\]
Define also 
\[
\dth^{\bG}(g)=\max\{\dth^{\bG}(x,g)\mid x\in\Bd(G),\ g\in\stab_G(x)\}.
\]
We simply write $\dth$ for $\dth^{\bG}$ is there is no confusion.
\end{definition}

We observe the following: 
\begin{enumerate}
\item  If $g\in \Z G_{x,0^+}$, $\dth(x,g)$ is the unique value $t$ so that $g\in \Z G_{x,t}\setminus \Z  G_{x,t^+}$. In most applications, it is possible to assume that $g\in G_{x,t}\setminus \Z G_{x,t^+}$ without loss of generality. In this case, we call $g$ \emph{noncentral mod $G_{x,t^+}$}. 

Likewise, when $\dth(g)=0$, one may assume that $g\in \stab_G(x)\setminus \Z G_{x,0^+}$ in most cases.
Note that $\stab_G(x)$ is compact since $\Bd(G)$ is an extended building. Again, we say $g$ is \emph{noncentral mod $G_{x,0^+}$} if $g\in \stab_G(x)\setminus \Z G_{x,0^+}$.
\item $\dth(y,g)\le\dth(g)$ for any $y\in \Bd(G)$, $g\in \stab_G(y)$ . 
\item
If $g\in \Z G_0$, $\dth(g)$ is the unique non negative real number $t$ such that 
$g\in \Z G_t\setminus \Z G_{t^+}$.
\item $\dth(x,g)=\dth(x,g')$ for all $g'\in g G_{x,t^+}$ where $\dth(x,g)=t$.
\item  $\dth(g)=\infty$ if and only if $g \in \Z \,\CaU\cap G$.
\item Let $g\in G^\rig$. If $g=\gamma u$ is the Jordan decomposition of $g$ with $\gamma\in G^\sms$ and $u\in\CaU$, we have $\dth(g)=\dth(\gamma)$.
\end{enumerate}

\begin{definition} [cf. Definition 6.1,\cite{AS}] \label{defn: good} \rm
For $\gamma\in G^\rig$, $\gamma$ is a \emph{$\bG$-good mod center element}
if there is a maximal $F$-torus $\bT$ which splits over a tamely ramified extension $E$ such that one of the the following holds:
\begin{enumerate}
\item $\gamma\in \Z T^c\setminus \Z T_{0^+}$  and the image of $\gamma$ in $\overline G$ is absolutely semisimple (see \cite{Hales} or \cite[Definition 4.11]{AS} for definition), where $T^c$ is the set of compact elements in $T$.
\item There is $t>0$ so that $\gamma\in \Z T_t\setminus \Z T_{t^+}$ with $\val(\alpha(\gamma)-1)=t$ or $\alpha(\gamma)=1$ for any $\alpha\in\Phi(\bT,\bG,E)$. 
\item $\gamma\in Z$.
\end{enumerate}
We will simply say $\gamma$ is $\bG$-\emph{good} of depth $t$ if either $\dth(\gamma)=0$ and $\gamma\in T^c$, or $\dth(\gamma)=t>0$ and $\gamma\in T_t$.
\end{definition}

\begin{remarks}\rm 
Keeping the situation as in the above definition, we observe the following:
\begin{enumerate}
\item
The depth of a good mod center element $\gamma$ is given as follows:
\[
\dth(\gamma)=
\begin{cases}
0&\textrm{in case (1)}\\
t&\textrm{in case (2)}\\
\infty&\textrm{in case (3).}
\end{cases}
\]
\item
If $\gamma\in T\setminus Z$ is a good mod center element of depth $t>0$ (resp. $0$), $\gamma=z\gamma_t$ for some $z\in Z$ and a good element $\gamma_t\in T_t\setminus T_{t^+}$ (resp. $\gamma\in T^c\setminus T_{0^+}$).
\item
Let $\gamma\in G^\sms$ and $\bG'=C_{\bG}(\gamma)$. Let $\dth^{\bG'}$ be the depth function defined on $G^{\prime\rig}$ as in Definition \ref{defn: depth}. In general $\dth^{\bG'}\neq\dth^{\bG}$ on $G^{\prime\rig}$. 
However, if $g\in G^{\prime\rig}\setminus Z(G')$ is $\bG$-good, it is also $\bG'$-good and $\dth^{\bG'}(g)=\dth^{\bG}(g)$.
\end{enumerate}
\end{remarks}

\subsection{\bf Hypotheses.}\label{subsec: hypos} \rm We collect here some assumptions that we need in this paper. We will be clear when each hypothesis is used. A lot of them are due to that we use results from \cite{AS} and \cite{De2}. Rather than repeating the statements of the hypotheses, we refer them directly to \emph{loc. cit.}.

\

\noindent{\bf Hypotheses (A)-(D)} These are Hypotheses (A)-(D) in \cite[\S2]{AS}.

\

\noindent{\bf Hypotheses (DB)} \rm These are the hypotheses in \S 2.1 and \S.4.3 in \cite{De2}.



\begin{hypo}\label{hyp: JD} \rm The Jordan decomposition is defined over $F$, i.e., for any $g \in G=\BG(F)$ and Jordan decomposition $g=s u$ of $g$ with $s, u \in \BG(\overline F)$, we have $s, u \in G$.
\end{hypo}

\begin{hypo}\label{hyp: unip} \rm For any $g \in G^\sms$, all the unipotent elements in $C_G(\g)$ are contained in $C_G(\g)^\circ$. 
\end{hypo}

\begin{hypo}\label{hyp: tame} \rm
Any torus in $\bG$ splits over a finite tamely ramified extension of $F$.
\end{hypo}

\begin{hypo}[Definition 6.3, \cite{AS}]\label{hyp: good} \rm
For any torus $\bS\subset\bG$ which splits over a tamely ramified extension $E$, and $r>0$, every nontrivial coset in $S_r/S_{r^+}$ contains a good element.
\end{hypo}


\begin{hypo}\label{conver}
For any $g \in G^\rig$, the orbital integral $O_g$ converges over $\mathbb C$. 
\end{hypo}


Hypothesis \ref{hyp: JD} holds if $F$ is of characteristic $0$, or if $p>\text{rank}_{ss}(G)+1$. But it fails when $F$ is of positive characteristic and $p$ is small. See \cite[Proposition 48 \& Remark 49]{Mc}. 
Hypothesis \ref{hyp: unip} automatically holds if $F$ is of characteristic $0$. If $F$ is of characteristic $p$, then it holds when $p$ is large but fails for some small $p$. For example, when $p=2$ and $C_G(\g)$ has two connected components, then any elements in $C_G(\g) \setminus C_G(\g)^\circ$ of order $2$ is unipotent. Hypothesis \ref{hyp: tame} holds if $p>\text{rank}_{ss}(G)$. Hypothesis \ref{hyp: good} holds when $\bG$ splits over a tamely ramified extension and $p$ does not divide the order of the Weyl group of $\bG$ (see \cite{Fintzen}). Hypothesis \ref{conver} holds if $F$ is of characteristic $0$ (see \cite{Rao}), and holds under some mild assumptions on $G$ and on $p$ if $F$ is of positive characteristic (see \cite[Theorem 61]{Mc} for the precise statement). 


\subsection{\bf Good elements and $\Bd(G)$}

Many results here can be found in \cite{AS}. For the Lie algebra versions, we refer to \cite{Kim, KMu, KMu2}. 

In the following four lemmas and a corollary, we let $\gamma\in T$ be a $\bG$-good mod center element of depth $t\ge0$. We also let $\bG':=C_{\bG}(\gamma)$. 

\begin{lemma}\label{lem: fixed pts} Suppose Hypotheses (A) and (B) hold. Define $\Bd(\gamma)$ as follows:
\[
\Bd(\gamma):=
\begin{cases}
\{x\in\Bd(G)\mid \gamma\in \Z\,\stab_G(x)\}&\textrm{if }t=0,\infty\\
\{x\in\Bd(G)\mid \dth(x,\gamma)=\dth(\gamma)\} &\textrm{if }t>0.
\end{cases}
\]
Then, we have $\Bd(\gamma)=\Bd(\bG',F)$.
\end{lemma}

\proof
If $\gamma\in Z$, clearly $\Bd(\gamma)=\Bd(G)$. Otherwise, without loss of generality, we may assume that $\gamma$ is $\bG$-good of depth $t$. Then, the lemma follows from \cite[Lemma 7.6]{AS}. 
\qed

\begin{lemma}\label{lem: approximation} Suppose Hypothesis (A) holds.
Let $x\in\Bd(G')$. Then, for $0\le t<s$ and $u\in G'_{x,t}\cap G'_{t^+}$, we have $^{G_{x,s-t}}(\gamma u G'_{x,s})=\gamma uG_{x,s}$.
\end{lemma}

\proof  
One may assume that $u$ is semisimple since $G^{\prime\sms}\cap uG'_{x,s}\cap G'_{t^+}\ne\emptyset$. Then, 
this follows from \cite[Corollary 7.5]{AS}. 
\qed

\begin{lemma}\label{lem: good element centralizer}
Suppose Hypothesis (C) holds. If $g\in\G$ is such that
$^g(\gamma G'_{t^+})\cap(\gamma G'_{t^+})\ne\emptyset$,
then $g\in G'=\bG'(F)$.
\end{lemma}

\proof Without loss of generality, we may assume $\gamma$ is $\bG$-good. Then, this is \cite[Lemma 7.1]{AS}.
\qed

\begin{lemma}\label{lem: good-building-2} 
Suppose Hypotheses (A) and (B) hold. Let $x\in\Bd(G)\setminus \Bd(G')$ and $u'\in G'_{\dth(\gamma)^+}$.
\begin{enumerate}
\item If $t=0$, $\gamma u'\not\in \Z\,\stab_G(x)$.
\item If $t>0$, then either $\gamma u'\not\in \Z\, \stab_G(x)$ or $\dth(x,\gamma u')<\dth(\gamma)$.
\end{enumerate} 
\end{lemma}

\proof When $u'=1$, this is \cite[Lemma 7.6]{AS}.

If $\gamma\in Z$, the statement is empty. We may assume that $\gamma$ is a $\bG$-good element.

(1) Suppose $\gamma u'\in \stab_G(x)$.  Then, we have $\gamma^{-1}x=u'x$ and $B:=\{\gamma^{-p^n}x\mid n\in \bbZ_{\ge0}\}=\{u'{}^{p^n}x\mid n\in \bbZ_{\ge0}\}\subset\Bd(G)$. Since $u'{}^{p^n}\rightarrow 1$ as $n\rightarrow\infty$, the set $B$ is finite and $x=u'{}^{p^n}x=\gamma^{-p^n}x$ for sufficiently large $n$. On the other hand, since $\gamma$ is absolutely semisimple and the order of $\gamma$ is relatively prime to $p$, there is an $n_\circ\in\bbZ_{>0}$ such that $\gamma^{-p^{n_\circ\ell}}=\gamma^{-1}$ for any $\ell\in\bbZ_{>0}$. 
Hence $\gamma x=x$, that is, $\gamma\in \textrm{Stab}_G(x)$, which is a contradiction to Lemma \ref{lem: fixed pts}.

\smallskip

(2) 
As in \cite[Lemma 7.6]{AS}, we may assume that $\gamma$ is split. 
Write $t=d(\gamma)$. Write $\gamma'$ for $\gamma u'$, and define
\[
\Bd(\gamma'):=\{x\in\Bd(\bG,F)\mid \dth(x,\gamma')\ge t\}.
\]
Note that $\Bd(\gamma')$ is convex and is a union of 
closures of chambers.

It is enough to show that $\Bd(\gamma')\subset\Bd(\bG',F)$. 

Suppose first that $\gamma$ is split and 
$\bG'$ is a $F$-Levi subgroup of $\bG$. Then, $t\in\mathbb N$.
Let $\bP$ be a $k$-parabolic subgroup of $\bG$ having 
Levi decomposition $\bP=\bG'\bN$. Let $\overline\bP$ be the 
parabolic subgroup opposite to $\bP$ with respect to 
this Levi decomposition $\bP=\bG'\bN$. Let $\overline\bN$ be
the unipotent radical of $\overline\bP$. 
Now, assume $\Bd(\gamma')\setminus\Bd(\bG',F)\ne\emptyset$,
and let $\chamberD$ be a chamber in $\Bd(\gamma')\setminus\Bd(\bG',F)$.
From the convexity of $\Bd(\gamma')$, we may assume that $\overline \chamberD$
shares a facet $F$ of codimension one 
with $\Bd(\gamma')\cap\Bd(\bG',F)$. 
Choose $y\in F$. 
From \cite[(2.4.1)]{AD}, there is 
$u\in G_{y,0}\cap N$ such that $uD\subset\Bd(\bG',F)$.
Then for $x\in \chamberD$, $ux\in uD\subset\Bd(\bG',F)$. 
Since $ux\in\Bd(\bG',F)$ \cite{MP2} and $uD$ is maximal, $G_{ux,t}$ has an Iwahori decomposition with respect to $(P,N)$, that is,
\[\tag{\dag}
G_{ux,t}=N_{ux,t}\cdot G'_{ux,t}\cdot \overline N_{ux,t} 
\]
where $N_{ux,t}=G_{ux,t}\cap N$ and $\overline N_{ux,t}=G_{ux,t}\cap \overline N$.
From this and the fact that $u\in N$, we can decompose $^u\!\gamma'$ as
\[\tag{\ddag}
^{u}\!\gamma'
\equiv\gamma^{\prime-1}({}^{u}\!\gamma')\gamma'\in G_{y,t}/G_{y,t^+}
\]
where $\gamma'\in G'\cap G_{y,t}=G'_{y,t}$ and 
$\gamma^{\prime-1}{}^{u}\!\gamma'\in N_{y,t}$. Since $G_{y,t^+}\subset G_{ux,t^+}\subset G_{ux,t}\subset G_{y,t}$, comparing $(\dag)$ and $(\ddag)$, we have 
$\gamma'\in G'\cap G_{ux,t}=G'_{ux,t}$ and 
$\gamma^{\prime-1}{}^{u}\!\gamma'\in N_{ux,t}$. Moreover,

\quad $(i)$ $u'\in G_{ux,t^+}$, \qquad
$(ii)$\  ${}^{u}\!\gamma'\in \gamma'G_{ux,t^+}
= \gamma G_{ux,t^+}$.

\noindent
$(i)$ follows from the fact that 
$uD$ is a chamber and $u'\in G_{ux,t}\cap G'_{t^+}$ (recall $t\in\mathbb N$).
$(ii)$ follows from $(i)$ and the fact that
$uD$ is a chamber and thus
$N_{ux,t}=N_{ux,t^+}$.
Then from Lemma  \ref{lem: approximation}, 
there is a $k\in\G_{ux,0^+}$
and an $u''\in G'_{ux,t^+}=G'\cap G_{ux,t^+}$ 
such that ${}^{u}\!\gamma'=\,^k(\gamma u'')$.
By Lemma \ref{lem: good element centralizer}, $u\in k\cdot G'\subset\G_{ux,0^+}G'$. 
Since $ux\in\Bd(\bG',F)$, we have 
$\G_{ux,0^+}=(N\cap\G_{ux,0^+})\cdot( G'\cap\G_{ux,0^+})
\cdot(\overline N\cap\G_{ux,0^+})$ \cite[(4.2)]{MP2}. So, we can conclude that 
$u\in\G_{ux,0^+}$ and $u(ux)=ux=x$. Then  $x\in\Bd(\bG',F)$,
which is a contradiction.
\qed

\smallskip

The following is a corollary of the proof of the above lemma:

\begin{corollary}\label{cor: good-building-2} Suppose Hypotheses (A) and (B) hold.
Let $u'\in G'_{\dth(\gamma)^+}$. Let $x\in\Bd(G)$.
\begin{enumerate}
\item If $t=0$ and $\gamma u'\in \Z\,\stab_G(x)$, both $\gamma$ and $u'$ are also in $\Z\,\stab_G(x)$.
\item If $t>0$ and $\gamma u'\in G_{x,t}$, both $\gamma$ and $u'$ are also in $\Z G_{x,t}$.
\end{enumerate}
\end{corollary}

\proof 
This follows from Lemmas \ref{lem: fixed pts} and \ref{lem: good-building-2}. \qed

\subsection{\bf Good products}

\begin{lemma}\label{lem: good-element-1} Suppose Hypothesis (C) holds.
Let $\bT$ be a maximal $F$-torus in $\bG$ which splits over a tamely ramified Galois extension $E$. Let 
$\gamma_{1},\cdots,\gamma_{n}\in T$ be $\bG$-good elements of depth $b=b_1, \cdots, b_n$ respectively. Let $\gamma=\gamma_z\gamma_1$ with $\gamma_z\in Z$. Let $\bH^0:=\bG$ and $\bH^i:=\C_{\bH^{i-1}}(\gamma_{i})$.
\begin{enumerate}
\item
Let $\gamma'\in\gamma H^1_{b^+}$. Then, $C_{\G}(\gamma')\subset H^1$. If $\gamma'$ is also $\bG$-good mod center of depth $b$, $C_G(\gamma')=H^1$.
\item
Suppose 
$b_1<b_2<\cdots <b_n$. Fix $i\in\{0,1,\cdots,n\}$ and $\gamma^i=\gamma_z\gamma_1\cdots\gamma_i$. Let $\gamma',\gamma''\in\gamma^iH^i_{b_i^+}$. 
If $^g\!\gamma'=\gamma''$ for some $g\in G$, then $g\in H^i$. 
\item
$\bH^i=\C_{\bH^{i-1}}(\gamma^i)=\C_{\bG}(\gamma^i)$. 
\item
$C_G(\gamma)\subset H^i$, $i=1,\cdots, n$.
\end{enumerate} 
\end{lemma}

\proof 
(1) If $\gamma_z=1$, the first statement is Lemma \ref{lem: good element centralizer}. Since $\gamma_z\in Z$, the statement remains valid for this case. For the second statement, since $\gamma'$ is also good of depth $b$, $H^1\subset C_G(\gamma')$. Combining this with the first statement, the second statement follows.

(2) We use induction on $i$. Since $\Phi(\bT, \bH^i,E)\subset\Phi(\bT,\bG,E)$ and $\gamma_i\in \bT_{b_i}$,
each $\gamma_i$, $i=1,\cdots,k$ is also $\bH^{i-1}$-good. When $i=1$, it is (1). Assume the statement is true for $i-1\ge 1$. Note that $\gamma',\gamma''\in \gamma^{i}H^{i}_{b_{i}^+}$. Suppose $^g\!\gamma'=\gamma''$ for some $g\in G$. Since $\gamma',\gamma''\in \gamma^{i-1}H^{i-1}_{b_{i-1}^+}$, we have $g\in H^{i-1}$ by the induction hypothesis and $^g(\gamma'(\gamma^{i-1})^{-1})=\gamma''(\gamma^{i-1})^{-1}$. Since $\gamma'(\gamma^{i-1})^{-1},\gamma''(\gamma^{i-1})^{-1}\in\gamma_{i}H^{i}_{b_{i}^+}$ and $\gamma_{i}$ is $\bH^{i-1}$-good and $\gamma^{i-1}\in Z(H^{i-1})$, $g\in H^{i}=C_{H^{i-1}}(\gamma_i)=C_{H^{i-1}}(\gamma^i)$ by (1). 

(3) The first equality follows since $\gamma^{i-1}\in Z(H^{i-1})$ and $\gamma^i=\gamma^{i-1}\gamma_i$. To prove the second equality, we use an induction. If $i=1$, it is trivial.  Suppose $i-1\ge1$. The inclusion $\C_{H^{i-1}}(\gamma^i)\subset\C_{G}(\gamma^i)$ is obvious. If $g\in \C_{G}(\gamma^i)$, we have $^g\gamma^i, \gamma^i\in\gamma^i H^i_{b_i^+}\subset\gamma^{i-1}H^{i-1}_{b_{i-1}^+}$. Then, $g\in H^{i-1}$ by (2). Hence, $\C_{G}(\gamma^i)\subset\C_{H^{i-1}}(\gamma^i)$.

(4) Note that $\gamma\in H^i$, $i=1,\cdots,n$. By (1), $C_G(\gamma)=C_{H^0}(\gamma)\subset H^1$. Suppose $C_G(\gamma)\subset H^{i-1}$ for $i\ge2$. Then, $C_G(\gamma) =C_{H^{i-1}}(\gamma)$. Since $\gamma^i$ is $\bH^{i-1}$-good mod center and $\gamma\in \gamma^{i}H^i_{b_i^+}$, $C_G(\gamma)=C_{H^{i-1}}(\gamma)\subset H^i$. 
\qed

\begin{lemma}\label{lem: good-element-2}
Let $\bT$ be a maximal $F$-torus in $\bG$ which splits over a tamely ramified Galois extension $E$. Let $\gamma_{1}, \gamma_{2}\in T$ be $\bG$-good mod center elements of depth $b_1,b_2$ respectively. Let $\bH^i=\C_{\bG}(\gamma_{i})$, $i=1,2$.
Suppose $b_1<b_2$ and $\gamma_2\in \Z(H^1)$. Then, $\gamma_{1}\gamma_{2}$ is also a $G$-good element of depth $b_1$.
\end{lemma}

\proof 
Write $\gamma=\gamma_{1}\gamma_{2}$. Let $\Phi:=\Phi(\bT,\bG,E)$ be the set of $E$-rational $\bT$-roots in $\bG$.  Let $\alpha\in\Phi$. Since $\bH^1\subset\C_{\bG}(\gamma_{2})$, $\alpha(\gamma_{1})=1$ implies $\alpha(\gamma_{2})=1$, thus $\alpha(\gamma_1\gamma_2)=1$. If $\alpha(\gamma_1)\neq 1$, since $\alpha(\gamma_1\gamma_2)-1=\alpha(\gamma_1)\alpha(\gamma_2)-\alpha(\gamma_2)+\alpha(\gamma_2)-1$, $\val(\alpha(\gamma_2))=0$ and $b_1=\val(\alpha(\gamma_{1})-1)<b_2\le\val(\alpha(\gamma_{2})-1)$, we have $\val(\alpha(\gamma)-1)=\min(\val(\alpha(\gamma_{1})-1),\val(\alpha(\gamma_{2})-1))=b_1$. 
Hence, $\gamma$ is $\bG$-good mod center of depth $b_1$.
\qed

\begin{proposition}\label{prop: good product} Suppose Hypothesis (C) and Hypothesis \ref{hyp: good} holds. Let $\bT$ be an $E$-split torus and $\gamma\in \Z T^c$. Then $\gamma$ is a product of good elements mod $r^+$ with decreasing centralizers in the following sense:
\begin{enumerate}
\item[(1)] $\gamma=\gamma_z\gamma_{1}\cdots\gamma_{k}\gamma_{r^+}$ where $z\in \Z$ and each $\gamma_i$ is $\bG$-good of depth $b_i$ with $b_1<b_2<\cdots< b_k\le r$ and $\gamma_{r^+}\subset T_{r^+}$, that is, $\dth(\gamma_{r^+})>r$, 
\item[(2)] $H^{1}\supsetneq H^{2}\supsetneq\cdots\supsetneq H^{k}$ where $H^{i}=C_G(\gamma_1\cdots \gamma_i)$.
\end{enumerate}
\end{proposition}

\proof 
If $\dth(\gamma)>r$, $\gamma=\gamma_z\cdot\gamma_{r^+}$ for some $\gamma_z\in \Z$ and $\gamma_{r^+}\in T_{r^+}$. 
Now, we assume that $\dth(\gamma)=\atm_1 \le r$. 

We first assume that $\gamma\in T_{a_1}$, that is, $\gamma$ is noncentral mod $T_{a_1^+}$.
By Hypothesis \ref{hyp: good}, $\gamma T_{\atm_1^+}$ contains a good element, say, $\tilde\gamma_{\atm_1}$ of depth $\atm_1$. Then $\gamma=\tilde\gamma_{\atm_1}(\gamma\tilde\gamma_{\atm_1}^{-1})$ with
$\atm_2=\dth(\gamma\tilde\gamma_{\atm_1}^{-1})>\dth(\gamma)$. We can choose $\tilde\gamma_{\atm_1}$ so that $\gamma\tilde\gamma_{\atm_1}^{-1}\in T_{\atm_2}$ by multiplying $\tilde\gamma_{\atm_1}$ with a central element if necessary.
Applying the above process for $\gamma\tilde\gamma_{\atm_1}^{-1}$, we find a $\bG$-good element 
$\tilde\gamma_{\atm_2}\in \gamma\tilde\gamma_{\atm_1}^{-1}\,T_{\atm_2^+}$
such that 
$\gamma
=\tilde\gamma_{\atm_1}\tilde\gamma_{\atm_2}
(\gamma(\tilde\gamma_{\atm_1}\tilde\gamma_{\atm_2})^{-1})$
and $\atm_3=\dth(\gamma(\tilde\gamma_{\atm_1}\tilde\gamma_{\atm_2})^{-1})
<\dth(\gamma\tilde\gamma_{\atm_1}^{-1})$.
Repeatedly, we have
\[
\gamma=\tilde\gamma_{\atm_1}\tilde\gamma_{\atm_2}\cdots\tilde\gamma_{\atm_m}\tilde\gamma_{r^+},
\]
where $\tilde\gamma_{\atm_i}$ is a $\bG$-good element 
of depth $\atm_i$ with
$\atm_1<\atm_2<\cdots<\atm_m\le r$ and $\dth(\tilde\gamma_{r^+})>r$.
This procedure is finite because $\dth(T_0)\subset\frac1{e(E/F)}\bbZ$.
Put $\atm_{m+1}:=r^+$ and $\tilde\gamma_{\atm_m+1}=\tilde\gamma_{r^+}$.

Set $\Set:=\{\atm_1,\atm_2,\cdots,\atm_{m+1}\}$, and
for $\atm,\btm\in\tbR$, set 
$\tilde\gamma_{\atm,\btm}:=\prod_{\atm\le\atm_j<\btm}\tilde\gamma_{\atm_j}$.
We find a subsequence 
$\btm_1<\btm_2<\cdots<\btm_n<\btm_{n+1}$
of $\Set$ as follows:
let $b_1:=a_1$ and $\bH^1:=\C_\bG(\tilde\gamma_{1})$.
Let $\btm_2$ be the maximal element in $\{\atm_2,\cdots,\atm_{m+1}\}$ 
with the property that 
if $a_j<\btm_2$, $\tilde\gamma_{a_j}\in\Z({H_{1}})$.
Note that $\bH^{1}=\C_\bG(\tilde\gamma_{\btm_1,\btm_2})$.
Let $\bH^{2}:=\C_{\bH^1}(\tilde\gamma_{\btm_2})$.
Then $\bH^{1}\supsetneq\bH^{2}$. 
Let $\gamma_{1}:=\tilde\gamma_{\btm_1,\btm_2}$.
Inductively, suppose $b_i$, $\bH^i$ and $\gamma_{{i-1}}$
are defined for $i\ge2$.
Let $\btm_{i+1}$ be the maximal element in 
$\{\atm_j\in\Set\mid\atm_j>\btm_i\}$ 
with the property that for any $\atm_j<\btm_{i+1}$, 
$\tilde\gamma_{a_j}\in \Z({H^{i}})$. 
Let $\bH^{i+1}:=\C_{\bH^{i}}(\tilde\gamma_{\btm_{i+1}})$
and $\gamma_{{i}}=\tilde\gamma_{{b_i},b_{i+1}}$.
We repeat the process until $\btm_{n+1}=\atm_{m+1}=r^+$.
Then each $\gamma_{i}$ is also a $\G$-good element of depth $\btm_i$
by Lemma \ref{lem: good-element-1}-(1), and 
$\bH^{i}=\C_{\bH^{i-1}}(\tilde\gamma_{\btm_i})
=\C_{\bH^{i-1}}(\gamma_{i})$, $i=1,\cdots, n$.
Now, one can easily check
\[\tag{$\ast$}
\gamma =\gamma_{1}\gamma_{2}\cdots\gamma_{n}\gamma_{r^+}
\]
satisfies the required properties. 

Now suppose $\gamma\in \Z T_{a_1}$. Then one can write $\gamma=\gamma_z \gamma'$ with $\gamma_z\in \Z$ and $\gamma'\in T_{a_1}$ noncentral mod $T_{a_1^+}$. Write $\gamma'=\gamma_1\cdots \gamma_n\gamma_{r^+}$ as in $(\ast)$. Then, $\gamma=\gamma_z\gamma_1\gamma_2\cdots\gamma_n\gamma_{r^+}$ satisfying the required properties.
\qed

\begin{definition}\label{defn: good product} \rm\ 
\begin{enumerate}
\item
We call the expression $\gamma_z\cdot\gamma_{1}\cdots\gamma_{k}\gamma_{r^+}$ of $\gamma$ in Proposition \ref{prop: good product} a \emph{good product of $\gamma$ mod $r^+$}. That is, $\gamma=\gamma_z\cdot\gamma_{1}\cdots\gamma_{k}\gamma_{r^+}$, where $\gamma_z\in Z$, $\gamma_{i}$ is $\bG$-good of depth $b_i$ with $b_1<\cdots b_k\le r$ and the sequence of centralizers $\bH^i(\gamma)=C_G(\gamma_z\gamma_{1}\cdots\gamma_{i})$ is strictly decreasing. In this case, we also write $\gamma_{\le r}:=\gamma_z\cdot\gamma_{1}\cdots\gamma_{k}$
\item
Let $\gamma=\gamma_z\gamma_{1}\cdots\gamma_{k}\gamma_{r^+}$ be a good product as in (1). 
Define $\bH^{\gamma,r}:=\bH^k(\gamma)$. We will often write $\bH^{\gamma}$ for $\bH^{\gamma,r}$ for simplicity. 
\end{enumerate}
\end{definition}

In \cite{AS}, $(\gamma_z,\gamma_1,\cdots,\gamma_k)$ is called an \emph{$r^+$-normal approximation} to $\gamma$.
The following is similar to \cite[Proposition 8.4]{AS}.



\begin{lemma}\label{lem: for dec} Suppose Hypothesis (C) holds.
Suppose $\gamma=\gamma_z\gamma_{1}\cdots\gamma_{k}$ and $\gamma=\gamma'_z \gamma'_{1}\cdots\gamma'_{k'}$ are two good products of $\gamma$ mod $r^+$ with $\dth(\gamma_i)=b_i$ and $\dth(\gamma'_i)=b'_i$. Write $\bH^{i}=C_{\bG}(\gamma_z\gamma_{1}\cdots\gamma_{i})$ and 
$\bH^{\prime i}=C_{\bG}(\gamma'_z \gamma'_{1}\cdots\gamma'_{i})$. Then, we have $k=k'$, $b_i=b_i'$ and $\bH^{i}=\bH^{\prime i}$. 
\end{lemma}

\proof

We have $\gamma_i, \gamma'_j\in C_G(\gamma)\subset (\cap_i H^i)\cap(\cap_j H^{\prime j})=H^k\cap H^{\prime k'}$ for $i\in\{z,1,\cdots,k\}$ and $j\in\{z,1,\cdots,k'\}$.

Note that $\dth(\gamma)=b_1=b'_1$. Since $\gamma_z'\gamma'_{1}\in\gamma_z\gamma_{1}H^1_{b_1^+}$, $\gamma_z\gamma_{1}\in\gamma'_z\gamma'_{1}H^{\prime1}_{b_1^+}$ and $\gamma_z\gamma_1$ and $\gamma'_z\gamma'_1$ are $\bG$-good mod center,  $H^{1}=H^{\prime 1}$ by Lemma \ref{lem: good-element-1}.
By induction, suppose that $b_j=b_j'$ and $\bH^{j}=\bH^{\prime j}$ for $1\le j\le i$. Write $\gamma^i=\gamma_{1}\cdots\gamma_{{i}}$ and $\gamma^{\prime i}=\gamma'_{1}\cdots\gamma'_{{i}}$.
Suppose $b_{i+1}<b'_{i+1}$. Then, $(\gamma^i)^{-1}\gamma^{\prime i}\in\gamma_{i+1} H^{i+1}_{b_{i+1}^+}$. Since $\bG$-good element $\gamma_{{i+1}}$ is also $\bH^{i}$-good,  we have $H^{i}\subset C_G((\gamma^i)^{-1}\gamma^{\prime i}) \subset H^{i+1}$  by Lemma \ref{lem: good-element-1} (2). This is a contradiction to $H^i\subsetneq H^{i+1}$. 
Hence, $b_{i+1}=b_{i+1}'$. Now we have 
\begin{enumerate}
\item[(i)] $\dth(\gamma_{{i+1}})=\dth(\gamma'_{{i+1}}\gamma^{\prime i}(\gamma^i)^{-1})=b_{i+1}$;

\item[(ii)] $\gamma^{\prime i}(\gamma^i)^{-1}\in \Z({H^{i}})$; 

\item[(iii)] $\gamma_{{i+1}}$, $\gamma'_{{i+1}}\gamma^{\prime i}(\gamma^i)^{-1}$ are $\bH^{i}$-good. 
\end{enumerate}
Combining (i)-(iii), it follows that $\C_\bG(\gamma^{i+1})=\C_{\bH^i}(\gamma_{{i+1}})=\C_{\bH^i}(\gamma'_{{i+1}})=\C_\bG(\gamma^{\prime i+1})$. Hence, $\bH^{i+1}=\bH^{\prime i+1}$. Similarly, one can show $k=k'$.
\qed

\subsection{\bf Decomposition of $G^\rig$} 
We first observe that $G^\rig=\Z \cdot\left(\cup_{x\in\Bd(G)}\stab_G(x)\right)$, and $\Z G_0\subset G^\rig$ where $G_0=\cup_{x\in\Bd(G)}G_{x,0}$. 

\begin{definition} \rm Let $\gamma,\gamma'\in G^\rig$.

\begin{enumerate}
\item
Suppose that $\gamma$ and $\gamma'$ are $\bG$-good mod center. We say that they are \emph{$\bG$-good $r^+$-equivalent mod center} and write $\gamma\overset{\rg}\sim \gamma'$ if there are $g\in G$ and a maximal torus $T$ such that $^g\!\gamma'\in \gamma T_{t^+}\subset T$ where $t=\min\{d(\gamma),r\}$.  We write a $\bG$-good $r^+$-equivalence class of $\gamma$ as $[\gamma]_\rg$.
Let 
\[
\begin{array}{ll}
[Z]^\rg_r&:=\{[z]_\rg\mid \bar z\in Z/Z_{r^+}\} \\ 
\CaS^\rg_r&:=[Z]^\rg_r\cup\{[\gamma]_{\rg}\mid \gamma \textrm{ is $\bG$-good mod center of depth } \dth(\gamma)\le r\}.
\end{array}
\]

\item
For  $\gamma,\gamma'\in G^\rig$, we say $\gamma$ and $\gamma'$ are {\it $r^+$-equivalent mod center} and write $\gamma\sim\gamma'$ if there are $g\in G$ and a maximal torus $T$ so that $^g\!\gamma'_{\le r}\in \gamma_{\le r}T_{r^+}\subset T$.
We write an $r^+$-equivalence class of $\gamma$ as $[\gamma]$, and 
let $\mathcal S_r$ be the set of $r^+$-equivalence classes of semisimple compact-modulo-center elements. 
\end{enumerate}
\end{definition}

Lemma \ref{lem: for dec} and the following lemma shows that the definition in (2) does not depend on the choice of truncation $\gamma_{\le r}$ and $\gamma'_{\le r}$.


\begin{proposition} \label{prop: dec Gc} Suppose Hypotheses (A)--(D) and 1-4 hold. Then, we have the following: 

\begin{enumerate}
\item[(1)] For $\gamma,\gamma'\in G^\rig$, 
$^G\!(\gamma H^\gamma_{r^+})\cap{}^G\!(\gamma' H^{\gamma'}_{r^+})\neq\emptyset$ if and only if $\gamma$ and $\gamma'$ are $r^+$-equivalent mod center. 
\item[(2)] If $\gamma\sim\gamma'$, $^G\!(\gamma H^\gamma_{r^+})={}^G\!(\gamma' H^{\gamma'}_{r^+})$
\item[(3)]
\[
\sqcup_{[\gamma]\in\mathcal S_r}{}^G\!\left(\gamma H^\gamma_{r^+}\right)=G^\rig. 
\]
\item[(4)] Each ${}^G\!\left(\gamma H^\gamma_{r^+}\right)$ is open and closed.
\end{enumerate}
\end{proposition}

\proof 
(1) For $\Leftarrow$, by Lemma \ref{lem: for dec}, we have $\gamma H^\gamma_{r^+}=\gamma_{\le r} H^\gamma_{r^+}$ and  $\gamma' H^{\gamma'}_{r^+}=\gamma'_{\le r} H^{\gamma'}_{r^+}$ independent of the choices of $\gamma_{\le r}$ and $\gamma'_{\le r}$. Since $\gamma_{\le r}T_{r^+}\subset H^\gamma$, Hence $\gamma'_{\le r}\in\, ^G\!(\gamma H^\gamma_{r^+})\cap{}^G\!(\gamma' H^{\gamma'}_{r^+})$.

For $\Rightarrow$ and (2), without loss of generality, one may assume that $\gamma H^\gamma_{r^+}\cap \gamma' H^{\gamma'}_{r^+}\neq\emptyset$. One may also assume that $\gamma=\gamma_{1}\cdots\gamma_{k}$ and $\gamma'=\gamma'_{1}\cdots\gamma'_{{k'}}$ with $b_k, b'_{k'}\le r$, that is, $\gamma_{r^+}=\gamma'_{r^+}=1$. Let $\delta\in \gamma H^\gamma_{r^+}\cap \gamma' H^{\gamma'}_{r^+}$. Then, $\delta=\gamma\cdot h=\gamma'\cdot h'$ with $h\in H^\gamma_{r^+}$ and $h'\in H^{\gamma'}_{r^+}$. Let $h=h_sh_u$ (resp. $h'=h'_sh'_u$) be the Jordan decomposition of $h$ in $H^\gamma$ (resp. $h'$ in $H^{\gamma'}$). Then, $\gamma h_sh_u$ and $\gamma'h'_sh'_u$ are two expressions of the Jordan decomposition of $\delta$. By the uniqueness of Jordan decomposition, $\gamma h_s=\gamma' h'_s$. Note $\dth(h_s)=\dth(h)>r$ and $\dth(h'_s)=\dth(h')>r$. By applying Lemma \ref{lem: for dec} to $\gamma h_s$, we have $H^{\gamma h_s}=H^\gamma=H^{\gamma'}$. Hence, $\gamma H^\gamma_{r^+}=\gamma h_s H^\gamma_{r^+}=\gamma' H^{\gamma'}_{r^+}$ and $\gamma\sim\gamma'$.

For (3), write $G_{\mathcal S_r}$ for $\sqcup_{[\gamma]\in\mathcal S_r}{}^G\!\left(\gamma H^\gamma_{r^+}\right)$. Clearly $G_{\mathcal S_r}\subset G^\rig$, and $G_{\mathcal S_r}$ is a disjoint union by (1). Conversely, for any $g\in G^\rig$, there is $x\in\Bd(G)$ and $z\in Z$ so that $gz\in\stab_G(x)$. By \cite[Lemma 2.38]{Spice}, we have   the topological Jordan decomposition of $gz=g_\as g_\tu$ with $g_\as$ absolutely semisimple and $g_\tu\in G'_{0^+}$  where $G'=C_G(g_\as)$. By Lemma \ref{lem: good-building-2}, $x\in\Bd(G')$. By Corollary \ref{cor: good-building-2}, we have $g_{\as}\in\stab_G(x)$. Let $g_{0^+}=g_sg_u$ be the Jordan decomposition of $g_\tu$ in $G'$. Let $g_s=g_{1}\cdots g_{k}g_{r^+}$ be a good product of $g_s$ mod $r^+$ with $\dth(g_i)=b_i$. Then, since $g_u$ commutes with $g_s$, $g_u\in C_{G'}(g_{1}\cdots g_{k})=H^{g_\as g_s}$ by Lemma \ref{lem: good-element-1}-(2) and thus $gz=g_\as g_sg_u\in g_\as g_s H^{g_\as g_s}_{r^+}$. Hence, $g\in G_{\mathcal S_r}$ and $G^\rig\subset G_{\mathcal S_r}$. 

For (4), let $\gamma=\gamma_z\gamma_1\cdots\gamma_k\gamma_{r^+}$ and $H^{i}$ be as in Proposition \ref{prop: good product}. We may assume $\gamma_{r^+}=1$. Let $g\in G$ and $h\in H^\gamma_{r^+}$ so that $^g\!(\gamma h)\in 
{}^G\!\left(\gamma H^\gamma_{r^+}\right)$. We may assume that $g=1$.  Let $y\in\Bd(H^\gamma)$ with $\gamma h\in \gamma H^\gamma_{y,r^+}\subset \gamma H^\gamma_{r^+}$. By Lemma \ref{lem: approximation}, we have ${}^{H^{{k-1}}_{y,(r-b_k)^+}}\!(\gamma_{k} H^\gamma_{y,r^+})=\gamma_{k}H^{{k-1}}_{y,r^+}$ and thus ${}^{H^{{k-1}}_{y,(r-b_k)^+}}\!(\gamma H^\gamma_{y,r^+})=\gamma H^{{k-1}}_{y,r^+}$ since $\gamma_{1}\cdots\gamma_{{k-1}}\in Z(H^{{k-1}})$. Inductively, setting $\bH^{0}=\bG$, 
we have ${}^{H^{{i-1}}_{y,(r-b_i)^+}}\!(\gamma H^\gamma_{y,r^+})=\gamma H^{{i-1}}_{y,r^+}$ for $i=1,\cdots, k$. Hence, $\gamma h\in\gamma G_{y,r^+}\subset {}^G\!\left(\gamma H^\gamma_{r^+}\right)$ and hence ${}^G\!\left(\gamma H^\gamma_{r^+}\right)$ is open. It is also closed since its compliment is open. 
\qed

\begin{corollary}\label{cor: dec Gc} Suppose Hypotheses (A)--(D), 1--4 hold. \ 
\begin{enumerate}
\item
For $g\in G^\rig$ with $\dth(g)<\infty$, we have $g=\gamma\cdot u$ for a $\bG$-good mod center element $\gamma$ of depth $\dth(g)$ and $u\in G^\gamma_{\dth(\gamma)^+}$ where $\bG^\gamma=\C_\bG(\gamma)$.
\item
We have
\[
G^\rig=\left(\sqcup_{[\gamma]_\rg\in\CaS^\rg_r\setminus[Z]^\rg_r} \,^G\!\left(\gamma G^\gamma_{\dth(\gamma)^+}\right)\right)
\sqcup \left(\sqcup_{[z]_\rg\in[Z]^\rg_r}\,zG_{r^+}\right).
\]
\end{enumerate}
\end{corollary}

\proof
(1)  Applying the above lemma when $r:=\dth(g)$, we have $g\in \gamma G^\gamma_{r^+}$ for a $\bG$-good mod center element $\gamma$ of depth $\dth(g)$. 

(2) By (1) $G^\rig=\left(\cup_{[\gamma]_\rg\in\CaS^\rg_r\setminus[Z]^\rg_r} \,^G\!\left(\gamma G^\gamma_{\dth(\gamma)^+}\right)\right)\cup \left(\cup_{[z]_\rg\in[Z]^\rg_r}\,zG_{r^+}\right)$.
To prove the disjointness, suppose $[\gamma]_\rg,[\gamma']_\rg\not\in[Z]_\rg$. If $^G\!\left(\gamma G^\gamma_{\dth(\gamma)^+}\right)\cap\, ^G\!\left(\gamma' G^{\gamma'}_{\dth(\gamma')^+}\right)\neq\emptyset$, $\dth(\gamma)=\dth(\gamma')$ and $^G\!\left(\gamma G^\gamma_{\dth(\gamma)^+}\right)=\, ^G\!\left(\gamma' G^{\gamma'}_{\dth(\gamma')^+}\right)$ follows from the above lemma by setting $r=\dth(\gamma)$. The other cases are easier.
\qed


\section{\bf Descents}\label{sec: descents}

\subsection{\bf Theorem A}\label{subsec: thmA}  \ 

From now on, we fix $r\in\bbR_{>0}$. For any $[\gamma]\in\mathcal S_r$, let $\CH({}^G\!(\gamma H^\gamma_{r^+}))$ be the $\BZ[\frac{1}{p}]$-submodule of $\CH$ consisting of functions supported in ${}^G\!(\gamma H^\gamma_{r^+})$ and let $\overline \CH({}^G\!(\gamma H^\gamma_{r^+}))$ be its image in $\overline \CH$. By Lemma \ref{domain} and Proposition \ref{prop: dec Gc}, we have $$\CH^\rig=\oplus_{[\gamma]\in\mathcal S_r} \CH({}^G\!(\gamma H^\gamma_{r^+})), \quad \overline \CH^\rig=\oplus_{[\gamma]\in\mathcal S_r} \overline \CH({}^G\!(\gamma H^\gamma_{r^+})).$$

Let $\CH_{r^+}({}^G\!(\gamma H^\gamma_{r^+})):=\CH_{r^+} \cap \CH({}^G\!(\gamma H^\gamma_{r^+}))$.
Then we have $\CH_{r^+}^\rig \supset \oplus_{[\gamma]\in\mathcal S_r} \CH_{r^+}({}^G\!(\gamma H^\gamma_{r^+})).$ Note that the intersection of a double coset of $G_{x, r^+}$ with a given $G$-domain ${}^G\!(\gamma H^\gamma_{r^+})$, in general, is not closed under the left (or equivalently, right) multiplication of $G_{x, r^+}$. Thus we have $$\CH_{r^+}^\rig \neq\oplus_{[\gamma]\in\mathcal S_r} \CH_{r^+}({}^G\!(\gamma H^\gamma_{r^+})).$$

Let $\overline \CH_{r^+}({}^G\!(\gamma H^\gamma_{r^+}))$ be the image of $\CH_{r^+}({}^G\!(\gamma H^\gamma_{r^+}))$ in $\overline \CH_{r^+}$. In other words, $\overline \CH_{r^+}({}^G\!(\gamma H^\gamma_{r^+}))$ is the $\bbZ[\frac1p]$-submodule of $\overline\CaH$ consisting of elements represented by functions in $\CH_{r^+}({}^G\!(\gamma H^\gamma_{r^+}))$.

The main purpose of this section is to show that we still have the desired direct sum decomposition of $\overline \CH_{r^+}^\rig=\oplus_{[\gamma]\in\mathcal S_r} \overline \CH_{r^+}({}^G\!(\gamma H^\gamma_{r^+}))$. 

\begin{definition}\rm \ 
\begin{enumerate} 
\item
For any $s\in\tbR_{\ge 0}$, we define
\[
\CaH^{G}_{s}=\sum_{x\in\Bd(G)}C_c(G/G_{x,s}).
\]
For $s,t\in\tbR$, with $0<t<s$, and $\gamma_z\in Z\pmod{Z_{r^+}}$, define
\begin{gather*}
\CaH^{G}_{t,s}(\gamma_z)=\sum_{x\in\Bd(G)}C_c\left(\left(\gamma_z\cdot G_{x,t}\right)/G_{x,s}\right),
\\
\CaH^{G,\flat}_{t,s}(\gamma_z)=\sum_{x\in\Bd(G)} C_c\left(\left(\gamma_z\cdot(G_{x,t}\cap G_{t^+})\right)/G_{x,s}\right).
\end{gather*}
We note that $\CH^{G,\flat}_{s, s^+}(\gamma_z)$ is spanned by $\mathbbm{1}_X$, where $X=\gamma_zgG_{x,s^+} \in \left.\gamma_zG_{x, s}\right/G_{x, s^+}$ for some $g\in G_{x,s}$ and $x \in \Bd(G)$ with $g G_{x,s^+}\subset G_{s^+}$ by \cite[Corollary 3.7.8 \& Corollary 3.7.10]{AD}. 
For simplicity, we will also write
\[
\CaH^{G}_{t,s}:=\CaH^{G}_{t,s}(1),\qquad\CaH^{G,\flat}_{t,s}(1)=\CaH^{G,\flat}_{t,s}.
\]

\medskip

\item
Let $\gamma=\gamma_z\gamma_{1}\cdots\gamma_{k}\gamma_{r^+}\in T$ be a good product of $\gamma$ with $\dth(\gamma_i)=b_i$.
\begin{enumerate}
\item[(i)] 
Let $\bH^\gamma=C_{\bG}(\gamma_{\le r})$. Define
\[
\CaH_{[\gamma]}^{G,\flat}:=\sum_{x\in\Bd(H^\gamma)} C_c\left(\left(\gamma_{\le r}\cdot(H^\gamma_{x,r}\cap H^\gamma_{r^+}) G_{x,r^+}\right)/G_{x,r^+}\right)
\]
In particular, for any $\gamma_z\in Z(G)$, we have 
\[
\CaH_{[\gamma_z]}^{G,\flat}:=\sum_{x\in\Bd(G)} C_c\left(\left(\gamma_z\cdot \left(G_{x,r}\cap G_{r^+}\right)G_{x,r^+}\right)/G_{x,r^+}\right).
\]
\item[(ii)]
If $[\gamma]_\rg:=[\gamma_z\gamma_{1}]_\rg\in\CaS^\rg_r$ with $\gamma_1$ $\bG$-good of depth $b_1\le r$, 
define
\[
\CaH_{[\gamma]_\rg}^{G,\flat}:=\sum_{x\in\Bd(H^1)} C_c\left(\left(\gamma_z\gamma_{1}\cdot(H^1_{x,b_1}\cap H^1_{b_1^+}) G_{x,r^+}\right)/G_{x,r^+}\right)
\]
where $\bH^1=C_{\bG}(\gamma_{1})$.  If $[\gamma]_\rg=[\gamma_z]_\rg$, let
\[
\CaH_{[\gamma_z]_\rg}^{G,\flat}:=\CaH_{[\gamma_z]}^{G,\flat}.
\]
Note that we  have
$\CaH_{[1]_\rg}^{G,\flat}=\CaH_{[1]}^{G,\flat}=\CaH^{G,\flat}_{r,r^+}(1).$
\end{enumerate}
\end{enumerate}
In all cases, we denote the image of each $\mathbb Z[\frac1p]$-submodule in the cocenter $\overline\CaH$ using $\overline{\phantom{\CaH}}$, e.g., $\overline\CaH_s^G$, $\overline\CaH^{G,\flat}_{t,s,}$, etc.
\end{definition}

\begin{theorem}\label{JD-H} Suppose Hypotheses (A)--(D) and 1-4 hold.
\begin{enumerate}
\item 
$\overline \CH_{r^+}^\rig=\oplus_{[\gamma]\in\mathcal S_r} \overline \CH_{r^+}({}^G\!(\gamma H^\gamma_{r^+}))$.
\item
For any $[\g] \in \CS_r$, $\overline \CH_{r^+}({}^G\!(\gamma H^\gamma_{r^+}))=\overline\CH^{G,\flat}_{[\g]}$.
 \end{enumerate}
\end{theorem}

We will prove the above theorem in the rest of this section. We first need some lemmas.

\subsection{\bf Some lemmas}

The following is \cite[Lemma 4.5.1]{De2}:

\begin{lemma}\label{lem: DeB-depth-descent} Suppose Hypotheses (DB) and Hypothesis \ref{hyp: unip} hold.
Let $x\in\Bd(\bG,F)$ and suppose $s<r$. Let $\bS\subset\bG$ be a maximal $k$-split torus of $\bG$ such that $x\in\Apt(\bS,F)$. If $u\in (\mathcal U G_{x,s^+}\cap (G_{x,s}\setminus G_{x,s^+}))$, then there exist $v\in{}^{G_x}\!(uG_{x,s^+})$ and $\lambda\in \bX_\ast(\bS,F)$ such that for sufficiently small $\epsilon>0$, we have 
\begin{enumerate}
\item $vG_{x,s^+}\subset G_{x+\epsilon\lambda, s^+}$ and
\item $vv'G_{x+\epsilon\lambda, r^+}\subset {}^{G_{x,(r-s)}}\!(vv'G_{x,r^+})$ for any $v'\in G_{x,s^+}$.
\end{enumerate}
\end{lemma}

\begin{definition}\rm  (\cite{BM}) 
For any $g\in G$, the \emph{displacement function} $\dpl_g:\Bd(G)\rightarrow\bbR$ is defined as $\dpl_g(x)=\dist(\bar x,g\bar x)$ where $\dist(\bar x,g\bar x)$ is the geodesic distance in the reduced building $\Bd(\overline G)$ between $\bar x$ and $g\bar x$ where $\bar x$ is the image of $x$ in $\Bd(\overline G)$. Define 
$\dpl(g):=\min\{\dpl_g(x)\mid x\in\Bd(G)\}$. 
For any subset $S\subset \Bd(G)$ with compact image in $\Bd(\overline G)$, 
define $\dpl_{S}(g):=\min\{\dpl_g(x)\mid x\in S\}$. Note that $\dpl_S$ is well defined since $S$ has a compact image in $\Bd(\overline G)$.
\end{definition}

We would also need the notion of generalized $r$-facets. In \cite{De1}, they are defined as certain subsets of the reduced building $\Bd(\overline G)$. One can define generalized $r$-facets on the extended building $\Bd(G)$ in a similar way:

\begin{definition} \rm (\cite{De1})
For $x\in\Bd(G)$, define
\begin{align*}
\face(x):
&=\{y\in\Bd(G)\mid \lieG_{x,\rtm}=\lieG_{y,\rtm}\textup{ and }
                          \lieG_{x,\rtm^+}=\lieG_{y,\rtm^+}\}\\
&=\{y\in\Bd(G)\mid G_{x,\rtm}=G_{y,\rtm}\textup{ and }
                          G_{x,\rtm^+}=G_{y,\rtm^+}\}\\
\Facet(\rtm):
&=\{\face(x)\mid x\in\Bd(G)\}
\end{align*}
An element in $\Facet(\rtm)$ is called a {\it generalized $\rtm$-facet}
in $\Bd(G)$. We will often write $\face$ for $\face(x)$ when there is no confusion.
Note that  the closure $\overline F^\ast$ of $F^\ast\in\face(r)$ has a compact image in $\Bd(\overline G)$.
For $\face=\face(x)\in\Facet(\rtm)$, define 
\[
\begin{array}{c}
\lieG_{\face}:=\lieG_{x,\rtm},\qquad
\lieG{}_{\face}^+:=\lieG_{x,\rtm^+}\ ,\\
\G_{\face}:=\G_{x,\rtm},\qquad
\G{}_{\face}^+:=\G_{x,\rtm^+}\ .
\end{array}
\]

\end{definition}

\begin{remarks}\label{rmks: disp properties} \rm
Let $y\in\Bd(G)$ and $\bS$ be a maximal $F$-split torus of $\bG$ with $y\in\Apt(S)$. Let $C$ be an alcove (0-facet of maximal dimension) with $y\in\overline C\subset\Apt(S)$. For $g\in G$, there are $n\in N_G(S)$ and $b_i\in G_C$ with $g=b_1nb_2$. Define $g'=\,^{b_1^{-1}}\!g=nb_2b_1$. Then, we have the following:
\begin{enumerate}
\item For $r\ge0$, since $b_1\in N_G(G_{y,r})\cap N_G(G_{y,r^+})$, we have $b_1\in\stab_G(F^\ast)$ 
where $F^\ast$ is the $r$-facet containing $y$.
\item We have $\chf_{gG_{F^\ast}^+}\equiv\chf_{g'G_{F^\ast}^+}\mod{[\CaH,\CaH]}$. Here for $X\subset G$, $\chf_X$ denotes the characteristic function with support $X$.
\item Since $y\in\overline C$, we have $g'y=nb_2b_1y=ny\in\Apt(S)$.
\item $\dpl_{g'}(y)=\dpl_g(y)\le \dpl_g(x)$ for all $x\in\overline F^\ast$.
\end{enumerate}
\end{remarks}

The proof of the following lemma is adapted from that of \cite[Corollary 4.2.9]{De2}. We include the proof for completeness.

\begin{lemma}\label{lem: descend by disp}
Let $\face\in\Facet(r)$ and $g\in G$. 
Suppose $\dpl(g)=0$ and $m:=\dpl_{\oface}(g)>0$. Then, there is a finite set $\{g_i\}$ and constants $c_i\in \bbZ[\frac1p]$ and $\face_i\in\Facet(r)$ such that
\begin{enumerate}
\item $\dpl_{\oface}(g)>\dpl_{\oface_i} (g_i)$ for each $i$, and
\item $\chf_{gG_{\face}^+}\equiv\sum_i c_i \chf_{g_iG_{\face_i}^+}\mod{[\CaH,\CaH]}$. 
\end{enumerate}
\end{lemma}

\proof We divide the proof in two cases.

\smallskip

\noindent
\underline{Case 1} There is $y\in\face$ with $\dpl_g(y)=\dpl_{\oface}(g)$.

\smallskip

Choose $\bS$ and $C$ as in Remarks \ref{rmks: disp properties}, and keep the notation from there.
Write $[y,ny]$ for the geodesic in $\Apt(S)$ between $y$ and $ny$. Observe that $[y,ny]\cap\oface=\{y\}$ (see the proof of \cite[Lemma 4.2.6]{De2} for details). 
Let $\face_1\in\Facet(r)$ be the first generalized $r$-facet that $(y,g'y]=(y,ny]$ passes through when traveling from $y$ to $ny$. Note that $\oface_1\cap\Apt\neq\emptyset$. Since $\face\subset\oface_1$ and thus $\face\cap\Apt\subset\oface_1\cap\Apt$. Note that $G^+_{\face}\subset G^+_{\oface_1}$. Let
\[
Q:=\{\psi\in\Psi(\bS,\bG,F)\mid\psi|(\face_1\cap \Apt)>r\textrm{ and }\psi|(\face\cap\Apt)=r\}.
\]
Then, 
\[
G^+_{F_1^\ast}=G^+_{F^\ast}\cdot\prod_{\psi\in Q}U_\psi
\]
where the product over $Q$ may be taken in any order. Fix $\psi\in Q$. Since $(n^{-1}\psi)(y)=\psi(ny)>r$, we have
$^{n^{-1}}\! U_\psi=U_{n^{-1}\psi}\subset G^+_{F^\ast}$. We also have $U^+_\psi\subset G^+_{F^\ast}$. 
By (2) of the above remarks, we have
\begin{align*}
\chf_{g\cdot G^+_{F^\ast}}&\equiv \chf_{g'\cdot G^+_{F^\ast}}\\
&\equiv c\cdot\sum_{\overline h\in(\prod_{\psi\in Q}U_\psi)/(\prod_{\psi\in Q}U_\psi^+)} \chf_{h^{-1}nb_2b_1G^+_{F^\ast}h}\\
&\equiv c\cdot\sum_{\overline h\in(\prod_{\psi\in Q}U_\psi)/(\prod_{\psi\in Q}U_\psi^+)} \chf_{g'G^+_{F^\ast} h}\\
&\equiv c \cdot\chf_{g'G^+_{F^\ast_1}}\mod{[\CaH,\CaH]}
\end{align*}
where the constant $c=\sharp\left((\prod_{\psi\in Q}U_\psi)/(\prod_{\psi\in Q}U_\psi^+)\right)^{-1}\in \bbZ[\frac1p]$.
Note that for all $z\in F^\ast_1\cap (y, g'y)\neq\emptyset$, we have from \cite[Lemma 4.2.1]{De2} that $\dpl_{g'}(z)<\dpl_{g'}(y)$. Hence, by (4) of the above remarks, we have
\[
\min_{x\in\overline{F^\ast_1}}\dpl_{g'}(x)<\dpl_{g'}(y)=\min_{x\in\overline{F^\ast}}\dpl_{g}(x).
\]

\noindent
\underline{Case 2} For all $x\in\face$, $\dpl_g(x)>\dpl_{\oface}(g)$.

\smallskip

Choose $y\in\oface\setminus\face$ such that $\dpl_{\oface}(g)=\dpl_g(y)$. There exists $\oface_1\in\Facet(r)$ such that $y\in\face_1$ and $\face_1\subset\oface$. Then, 
\[
\chf_{gG_{\face}^+}=\sum_{\alpha\in G_{\face}^+/G_{\face_1}^+}\chf_{g\alpha G_{\face_1}^+}.
\]
Note that for all $\alpha\in G_{\face}^+$, we have $\dpl_g(x)=\dpl_{g\alpha}(x)$ for all $x\in\oface$ and 
$\dpl_{g\alpha}(y)=\dpl_{\oface_1}(g\alpha)$ for all $\alpha\in G_{\face}^+$.
Now, one can apply Case 1 to each summand $\chf_{g\alpha G^+_{\face_1}}$, $\alpha\in G_{\face}^+/G^+_{\face_1}$, and $\oface_1$. \qed

\subsection{\bf Descents}

\begin{proposition}\label{prop: descent}
Let $g G_{y,r^+}\subset G^\rig$. Write $s:=\dth(y,g)$ and $t:=\dth(g)$.  
There exist a finite indexing set $\{i\}$, $\{g_i\}\subset G$ and $c_i\in\bbZ[\frac1p]$ such that $\chf_{gG_{y,r^+}}\equiv\sum_i c_i \chf_{g_iG_{y_i,r^+}}\pmod{[\CaH,\CaH]}$ with $\dth(g_i)=\dth(g_i,y_i)\le r$ or $g_i\in Z\cdot(G_{r^+}\cap G_{y_i,r})$. 
\end{proposition}

Note that $s\le t$. Note also that if $s<r$, we have $\dth(y,g)=\dth(y,g')$ for all $g'\in g G_{y,r^+}$.

\proof We prove the statement in 3 cases below. Without loss of generality, we may assume that $g$ is compact.

\medskip

\underline{Case 1} $s=t=0$ or $s>r$. 

Done since $\chf_{gG_{y,r^+}}$ already satisfies the required condition. In particular, when $s>r$, $gG_{y,r^+}=z G_{y,r^+}$ for some $z\in Z$.

\medskip

\underline{Case 2} $\dth(g)=0$ and $g\in Z\,\stab_G(x)\setminus Z\,\stab_G(y)$.

\smallskip

In this case, $\dpl_g(y)>0$. Let $\face\in\Facet(r)$ with $y\in\oface$.

If $\dpl_{\oface}(g)=0$, from Case 1 in Lemma \ref{lem: descend by disp}, we may assume that there is $z\in\oface$ such that $\dpl_g(z)=0$. Then,  $gG_{y,r^+}=gG_{z,r^+}\subset\stab_G(z)$, which reduces to the Case 3 below. 

Now, let $\dpl_{\oface}(g)>0$. By applying Lemma \ref{lem: descend by disp} repeatedly, we can write
\[
\chf_{gG_{y,r^+}}=\chf_{gG^+_{\face}}=\sum_i c_i\chf_{g_iG^+_{\face_i}}
\]
with $\{i\}$ a finite set, $\face_i\in\Facet(r)$, and $\dpl_{\oface_i}(g_i)=0$ for all $i$.
More precisely, applying Lemma \ref{lem: descend by disp} repeatedly, we find a sequence of triples $(g_j,\face_j,y_j)\in G\times\Facet(r)\times\Bd(G)$, $j\in\mathbb N$ such that $y_j\in\face_j$ and 
\[
\dpl_{g_j}(y_j)=\dpl_{\face_j}(g_j)>\dpl_{\face_{j+1}}(g_{j+1})=\dpl_{g_{j+1}}(y_{j+1})>0.
\]
Since $\dpl_{g_j}(y_j)$ is a discrete decreasing sequence in $\bbR_{\ge0}$, for sufficiently large $j$, $\dpl_{g_j}(y_j)=0$ (see the proof of \cite[Theorem 4.1.4-(1)]{De2}).

\medskip

\underline{Case 3} $0\le s\le t$.  

If $t=s$, it is done.

Now, suppose that $t>s$. Then, $g\in G_{y,s}\cap \mathcal UG_{y,s^+}$. 
We claim that there is a finite set $\{i\}$ such that 
\begin{align}\label{tag: claim}
\chf_{g G_{y,r^+}}\equiv\sum c_i\chf_{g_i G_{x_i,r^+}}\mod{[\CaH,\CaH]}
\end{align}
with $c_i\in\bbZ[\frac1p]$ and $\dth(x_i,g_i)>s$. 

Let $F^\ast\subset\Bd(G)$ be the generalized $r$-facets with $y\in F^\ast$. Let $\bS$ be maximal maximally $F$-split tori in $\bG$ so that $F^{\ast}\subset\Apt(\bS,F)$. Since $\dth(\ ,g)$ is continuous on $\Bd(G)$, for fixed $g\in G$, $\dth(\ ,g)$ attains its maximum on $\overline F^{\ast}$ since the image of $\overline F^\ast$ in the reduced building is compact. That is, there is $s\in\mathbb R_{\ge0}$ and $x\in\overline F^\ast$ such that
\begin{align*}
s=\dth(x,g)\ge \dth(w, g)
\end{align*}
for all $w\in\overline F^\ast$.

Since $x\in\overline F^\ast$, we have $\lieG_{x,r^+}\subset\lieG^{+}_{F^\ast}\subset\lieG_{F^\ast}\subset\lieG_{x,r}$. Therefore,
\[
\chf_{gG_{y,r^+}}=\sum_{\overline\alpha\in G_{y,r^+}/ G_{x,r^+}}\chf_{g\cdot \alpha G_{x,r^+}}.
\]
Note that each $\alpha\in G_{x,r}$ and $\dth(x,g)=\dth(x,g\alpha)=s$. 

Now, for each $g\alpha$, we will show that there is a $z_\alpha\in \Bd(G)$ such that $\chf_{g\alpha G_{x,r^+}}$ is a linear combination of characteristic functions of the form $\chf_{g_\alpha G_{z_\alpha,r^+}}$ with $g_\alpha\in G_{z_\alpha,s^+}$ (see (\ref{tag: step1}) below). 

Since we are treating each $g\alpha$, for simplicity of notation, we may write $u$ for $g\alpha$.
Note that $u\in G_{x,s}\cap G_{s^+}\subset \mathcal U G_{x,s^+}$. Then, we can find $v\in{}^{G_x}\!(uG_{x,s^+})$ and $\lambda$ as in Lemma \ref{lem: DeB-depth-descent} so that for sufficiently small $\epsilon>0$, we have (i) $vG_{x,s^+}\subset G_{x+\epsilon\lambda, s^+}$ and
(ii) $vv'G_{x+\epsilon\lambda, r^+}\subset {}^{G_{x,(r-s)}}\!(vv'G_{x,r^+})$ for any $v'\in G_{x,s^+}$. 
Fix $\epsilon$ satisfying (i) and (ii). Write $z=x+\epsilon\lambda$. Let 
\[
\mathfrak a:=\{h(v)h^{-1}(v)^{-1}\mid h\in G_{x,(r-s)/(r-s)^+}\}G_{x,r^+}.
\]
We have $\mathfrak a\subset G_{x,r}\subset G_{x,s^+}$ and $G_{z,r^+}\subset\mathfrak a$. Then,
\begin{align*}\tag{2}\label{tag: step1}
\chf_{u G_{x, r^+}}&\equiv\chf_{vG_{x,r^+}} \equiv c\cdot\sum_{\overline h\in G_{x,(r-s)}/G_{x,(r-s)^+}} \chf_{h vh^{-1}G_{x,r^+}}\\
&\equiv c\cdot \chf_{v\cdot\mathfrak a} \equiv c\cdot\sum_{\overline\beta\in\mathfrak a/G_{z,r^+}}\chf_{ v\beta G_{z,r^+}} \mod [\CH, \CH],
\end{align*}
where $c=\left(\sharp(G_{x,(r-s)}/G_{x,(r-s)^+})\right)^{-1}\in\bbZ[\frac1p]$.
For all $\beta\in\mathfrak a$, we have $\chf_{ v\beta G_{z,r^+}}\in C (G_{z,s^+}/G_{z,r^+})$, and the claim is now proved.

Now one can repeat the process for $\chf_{g_i G_{x_i,r^+}}$ in (\ref{tag: claim}) until each coset satisfy $\dth(g_i)=\dth(x_i,g_i)$ or $\dth(x_i,g_i)\ge r$. This is a finite process as in the proof of 
\cite[Theorem 4.1.4-(1)]{De2}. We omit the details.
\qed

\begin{remarks}\rm
We observe the following: Suppose $g G_{y,t^+}$ satisfies $t:=\dth(g)=\dth(y,g)$. Then, $y\in \Bd(G')$ where $G'=C_G(\gamma)$. Write $g=\gamma\cdot u$ such that $\gamma$ is a  $\bG$-good mod center element and $u\in G'_{d(\gamma)^+}$ (see  Corollary \ref{cor: dec Gc}). Then,
\begin{enumerate}
\item $\dth(g)=\dth(\gamma)=\dth(y,g)=\dth(y,\gamma)$,
\item $y\in\Bd(\bG',F)$,
\item $\gamma\in Z G'_{y,t}$, \ $u\in G'_{y,t}\cap G'_{t^+}$,
\item $\chf_{g G_{y,r^+}}\in C_c\left(\left(\gamma\cdot(G'_{y,t}\cap G'_{t^+}) G_{y,r^+}\right)/G_{y,r^+}\right)$.
\end{enumerate}
The first equality $\dth(g)=\dth(\gamma)$ follows from Lemma \ref{lem: good-building-2}.
(2) is Lemma \ref{lem: fixed pts} and (3) is Corollary \ref{cor: good-building-2}.
\end{remarks}

The following is a corollary of Proposition \ref{prop: descent}:

\begin{corollary}\label{3.3.3}
$\overline\CaH_{r^+}^{\rig}=\bigoplus_{[\gamma]_\rg\in\CaS_r^\rg}\overline\CaH_{[\gamma]_\rg}^{G,\flat}$.
\end{corollary}

\proof 
It follows from Corollary \ref{cor: dec Gc}, Proposition \ref{prop: descent} and  Lemma \ref{domain}.
\qed

\subsection{\bf Descent via induction}

\begin{proposition}\label{good-i}
Let $\g=\gamma_z\gamma_1$ be a $\bG$-good mod center element of depth $t \ge 0$, where $\gamma_z\in Z(G)$ and $\gamma_1$ is $\bG$-good of depth $t$. Let $s \in \tilde \BR$ with $s>t$. We set $G'=C_G(\g)$. Then the map $\mathbbm{1}_{h G'_{x, s}} \mapsto \frac{\mu_{G'}(G'_{x, s})}{\mu_G(G_{x, s})} \mathbbm{1}_{\g h G_{x, s}}$ for $x \in \Bd(G')$ and $h \in G'_{x, t} \cap G'_{t^+}$ induces a well-defined map 
$$\bar i_{\g, s}: \overline\CH^{G', \, \flat}_{t, s} \to \overline\CH^{G}_{t, s}(\gamma_z).$$ 
Moreover, for any $s' \in \tilde \BR$ with $s' \ge s$, we have the following commutative diagram 
\[\xymatrix{
\overline\CH^{G', \, \flat}_{t, s} \ar[r]^-{\bar i_{\g, s}} \ar@{^{(}->}[d] & \overline\CH^{G}_{s} \ar@{^{(}->}[d] \\
\overline\CH^{G', \, \flat}_{t, s'} \ar[r]^-{\bar i_{\g, s'}} & \overline\CH^{G}_{s'}.}
\]
\end{proposition}

\begin{remark}\rm
Note that the elements $\mathbbm{1}_{h G'_{x, s}}$ for $x \in \Bd(\g)$ and $h \in G'_{x, t} \cap G'_{t^+}$ are not linearly independent in $\CaH$. Thus the map $\mathbbm{1}_{h G'_{x, s}} \mapsto \frac{\mu_{G'}(G'_{x, s})}{\mu_G(G_{x, s})} \mathbbm{1}_{\g h G_{x, s}}$  may not give a well-defined map from $\CH^{G', \, \flat}_{t, s}$ to $\CH^{G}_{s}$. However, we will see that it induces a well-defined map on the cocenter. 
\end{remark}

\begin{proof}[Proof of Proposition \ref{good-i}]
Let $x \in \Bd(\g)$ and $h \in G'_{x, t} \cap G'_{t^+}$. Let $\e=s-t$. By Lemma \ref{lem: approximation}, for any $k \in G_{x, s}$, there exists $g \in G_{x, \e}$ such that ${}^g(\g h k) \in \g h G'_{x, s}$. Note that $(G_{x, \e}, G'_{x, s}) \subset G_{x, s+\e}$. Then $G'_{x, s} G_{x, s+\e}$ is a subgroup of $G$ and ${}^g (G'_{x, s} G_{x, s+\e})=G'_{x, s} G_{x, s+\e}$. We have ${}^g(\g h k G'_{x, s} G_{x, s+\e})=\g h G'_{x, s} G_{x, s+\e}$ and $$\mathbbm{1}_{\g h k G'_{x, s} G_{x, s+\e}} \equiv \mathbbm{1}_{\g h G'_{x, s} G_{x, s+\e}} \mod [\CH, \CH].$$

Similarly, for any $n \in \BN$ and $k \in G_{x, s+n \e}$, we may write $G'_{x, s}=\sqcup_l h_l G'_{x, s+n \e}$ for $h_l \in G'_{x, s}/G'_{x,s+n\epsilon}$. Since $G'_{x, s+n \e} \subset G_{x, s+n \e}$, we have $h k G'_{x, s}=\sqcup_l h h_l k_l G'_{x, s+n \e}$ for $k_l=h_l^{-1} k h_l \in G_{x, s+n\e}$. Note that $h h_l \in G'_{x, t} \cap G'_{t^+}$. By Lemma \ref{lem: approximation}, there exists $g_l \in G_{x, (n+1) \e}$ such that ${}^{g_l}\!(\g h h_l k_l) \in \g h h_l G'_{x, s+n \e}$. Note that $(G_{x, (n+1) \e}, G'_{x, s+n \e}) \subset G_{x, s+(2n+1) \e}$. Then $G'_{x, s+n \e} G_{x, s+(2n+1) \e}$ is a subgroup of $G$ and ${}^{g_l} (G'_{x, s+n \e} G_{x, s+2 (n+1)\e})=G'_{x, s+n \e} G_{x, s+(2n+1)\e}$. We have ${}^{g_l} (\g h h_l k_l G'_{x, s+n \e} G_{x, s+(2n+1)\e})=\g h h_l G'_{x, s+n \e} G_{x, s+(2n+1)\e}$ and \begin{align*} \mathbbm{1}_{\g h k G'_{x, s} G_{x, s+(2n+1)\e}} &=\sum_l \mathbbm{1}_{\g h_l k_l G'_{x, s+n \e} G_{x, s+(2n+1)\e}} \\ & \equiv \sum_l \mathbbm{1}_{\g h h_l G'_{x, s+n \e} G_{x, s+(2n+1)\e}} \\ &\equiv\mathbbm{1}_{\g h G'_{x, s} G_{x, s+(2n+1) \e}} \mod [\CH, \CH].\end{align*} 

In particular, we have 
\begin{align*} 
\mathbbm{1}_{\g h  G_{x, s}} & \equiv \frac{\mu_G(G_{x, s})}{\mu_G(G'_{x, s} G_{x, s+\e})} \mathbbm{1}_{\g h G'_{x, s} G_{x, s+\e}} \\ & \equiv \cdots \\ &  \equiv \frac{\mu_G(G_{x, s})}{\mu_G(G'_{x, s} G_{x, s+(2n+1)\e})} \mathbbm{1}_{\g h G'_{x, s} G_{x, s+(2n+1)\e}} \mod [\CH, \CH]. 
\end{align*}

Moreover, for any open compact subgroup $K$ with $G_{x, s+(2n+1) \e} \subset K \subset G_{x, s+n \e}$, we have $\mathbbm{1}_{\g h  G'_{x, s} K} \equiv \frac{\mu_G(G'_{x, s} K)}{\mu_G(G'_{x, s} G_{x, s+(2n+1)\e})} \mathbbm{1}_{\g h G'_{x, s} G_{x, s+(2n+1) \e}} \mod [\CH, \CH]$ and hence 
\[
\tag{a} \mathbbm{1}_{\g h  G_{x, s}} \equiv \frac{\mu_G(G_{x, s})}{\mu_G(G'_{x, s} K)} \mathbbm{1}_{\g h G'_{x, s} K} \mod [\CH, \CH].
\]
Now suppose that $\sum_{i} a_i \mathbbm{1}_{h_i  G'_{x_i, s}}=0 \in \CH^{G', \, \flat}_{t, s}$. We may choose a sufficiently large $n \in \BN$ such that the subgroup generated by $G_{x_i, s+(2n+1) \e}$ for all $i$ is contained in $K:=\cap_j G_{x_j, s+n \e}$. Then, $G_{x_i, s+(2n+1) \e} \subset K \subset G_{x_i, s+n \e}$ for all $i$. By definition, $\sum_{i} a_i \mathbbm{1}_{h_i  G'_{x_i, s}}$ is mapped to $\sum_{i} a_i \frac{\mu_{G'}(G'_{x_i, s})}{\mu_G(G_{x_i, s})} \mathbbm{1}_{\g h_i  G_{x_i, s}}$. By (a), $$\sum_{i} a_i \frac{\mu_{G'}(G'_{x_i, s})}{\mu_G(G_{x_i, s})} \mathbbm{1}_{\g h_i  G_{x_i, s}} \equiv \sum_{i} a_i \frac{\mu_{G'}(G'_{x_i, s})}{\mu_G(G'_{x_i, s} K)} \mathbbm{1}_{\g h_i  G'_{x_i, s} K} \mod [\CH, \CH].$$

Note that $\frac{\mu_{G'}(G'_{x_i, s})}{\mu_G(G'_{x_i, s} K)}=\frac{\mu_{G'}(G' \cap K)}{\mu_G(K)}$ for all $i$. Moreover, we have $\sum_{i} a_i \mathbbm{1}_{\g h_i  G'_{x_i, s} K}=0$ as $\sum_{i} a_i \mathbbm{1}_{h_i  G'_{x_i, s}}=0$. Therefore, under the map in the proposition, the image of $\sum_{i} a_i \mathbbm{1}_{h_i  G'_{x_i, s}}$ equals to $0$. Hence it gives a well-defined map from $i_{\g, s}: \CH^{G', \, \flat}_{t, s} \to \overline\CH^G_s$. 

Now let $\e'=s'-t$. Let $f \in \CH^{G', \, \flat}_{t, s} \subset \CH^{G', \, \flat}_{t, s'}$. We show that $i_{\g, s}(f)=i_{\g, s'}(f)$. 

Suppose that $f=\sum_i a_i \mathbbm{1}_{h_i G'_{x_i, s}}$. We choose an open compact subgroup $K$ such that there exists $n, n' \in \BN$ with $G_{x_i, s+(2n+1) \e} \subset K \subset G_{x_i, s+n \e}$ and $G_{x_i, s'+(2n'+1) \e'} \subset K \subset G_{x_i, s'+n' \e'}$ for all $i$. By (a), we have that 
$$i_{\g, s}(f)
\equiv\sum_i a_i \frac{\mu_{G'}(G' \cap K)}{\mu_G(K)} \mathbbm{1}_{\g h_i  G'_{x_i, s} K}\mod{[\CH, \CH]}.
$$ 

We may write $f$ as $f=\sum_i \sum_{g' \in G'_{x_i, s}/G'_{x_i, s'}} a_i \mathbbm{1}_{h_i g' G'_{x_i, s}} \in \CH^{G', \, \flat}_{t, s'}$. Then by (a), we have 
\begin{align*} 
i_{\g, s'}(f) &\equiv\sum_i \sum_{g' \in G'_{x_i, s}/G'_{x_i, s'}} a_i \frac{\mu_{G'}(G' \cap K)}{\mu_G(K)} \mathbbm{1}_{\g h_i g' G'_{x_i, s'} K} \\ 
&\equiv\sum_i a_i \frac{\mu_{G'}(G' \cap K)}{\mu_G(K)} \mathbbm{1}_{\g h_i G'_{x_i, s} K} \equiv i_{\g, s}(f) \mod{[\CaH,\CaH]}.
\end{align*}

It remains to show that the map $i_{\g, s}: \CH^{G', \, \flat}_{t, s} \to \overline\CH^G_s$ factors through $\overline\CH^{G', \, \flat}_{t, s}$. 

Let $f \in \CH^{G', \, \flat}_{t, s} \cap [\CH^{G'}, \CH^{G'}]$. Then by definition, the support of $f$ is contained in the $G'$-domain $G'_{t^+}$. By Lemma \ref{f-g-f}, $f=\sum_i (f_i-{}^{g_i}\! f_i)$, where $f_i \in \CH(G')$ with support in $G'_{t^+}$ and $g_i \in G'$. Let $s' \in \tilde \BR$ with $s' \ge s$ and that $f_i \in \CH^{G', \, \flat}_{t, s'}$ for all $i$. Then $i_{\g, s}(f)=i_{\g, s'}(f)$. It remains to prove that for any $f' \in \CH^{G', \, \flat}_{t, s'}$ and $g \in G'$, we have $i_{\g, s'}(f')=i_{\g, s'}({}^g\! f')$. 

It suffices to consider the case where $f=\mathbbm{1}_{h G'_{x, s'}}$, where $x \in \Bd(\g)$ and $h \in G'_{x, t} \cap G'_{t^+}$. By definition, 
\begin{align*} 
i_{\g, s'}(\mathbbm{1}_{{}^g (h G'_{x, s'})}) 
&\equiv i_{\g, s'}(\mathbbm{1}_{g h g^{-1} G'_{gx, s'}}) \equiv \frac{\mu_{G'}(G'_{gx, s'})}{\mu_G(G_{g x, s'})} \mathbbm{1}_{\g g h g^{-1} G_{gx, s'}}\\ 
&\equiv\frac{\mu_{G'}(G'_{x, s'})}{\mu_G(G_{x, s'})} \mathbbm{1}_{\g h G_{x, s'}} \equiv i_{\g, s'}(\mathbbm{1}_{h G'_{x, s'}}) \mod{[\CaH,\CaH]}.
\end{align*} 
This finishes the proof. 
\end{proof}

\begin{theorem} \label{thm: bar i}
Let $\g \in G^\rig \cap G^\sms$ with a good product $\gamma=\gamma_z\gamma_{b_1}\cdots\gamma_{b_k}\gamma_+$, where each $\gamma_{b_i}$ is $\bG$-good of depth $b_i$ and $b_1<b_2<\cdots< b_k\le r$. Then the map $\mathbbm{1}_{h H^\gamma_{x, r^+}} \mapsto \frac{\mu_{H^\gamma}(H^\gamma_{x, r^+})}{\mu_G(G_{x, r^+})} \mathbbm{1}_{\g_{\le r} h G_{x, r^+}}$ for $x \in \Bd(\g_{\le r})$ and $h \in H^\gamma_{x, r} \cap H^\gamma_{r^+}$ induces a well-defined map $$\bar i_{\g, r^+}: \overline\CH^{{H^\gamma}, \, \flat}_{r, r^+} \to \overline\CH^{G}_{b_1, r^+}(\gamma_z).$$ 
In particular, the map $\bar i_{\g, r^+}$ is independent of the good product expression $\gamma=\gamma_z\gamma_{b_1}\cdots\gamma_{b_k}\gamma_+$.
\end{theorem}

\begin{proof} Without loss of generality, we may assume $\gamma=\gamma_{\le r}$.
Set $\bH^i=C_{\bG}(\gamma_z\g_{b_1} \cdots \g_{b_i})$. Then by Proposition \ref{good-i}, we have well-defined maps 
\begin{gather*} 
\bar i^{H^{k-1}}_{\g_{b_k}, r^+}: \overline\CH^{{H^\gamma}, \, \flat}_{b_k, r^+} \to \overline\CH^{H^{k-1}}_{b_k, r^+}, 
\\ 
\bar i^{H^{k-2}}_{\g_{b_{k-1}}, r^+}: \overline \CH^{H^{k-1}, \, \flat}_{b_{k-1}, r^+} \to \overline \CH^{H^{k-2}}_{b_{k-1}, r^+}, \\ \cdots \\ 
\bar i^{G}_{\g_{b_1}, r^+}: \overline \CH^{H^1, \, \flat}_{b_1, r^+} \to \overline \CH^{G}_{b_1, r^+}(\gamma_z).
\end{gather*}

As $b_{i+1}<b_i$ for any $i$, we have that $\overline \CH^{H^i}_{b_{i+1}, r^+} \subset \overline \CH^{H^i, \, \flat}_{b_i, r^+}$. Set $\bar i_{\g_{b_1}, \g_{b_2}, \cdots, \g_{b_k}, r^+}=\bar i^{G}_{\g_{b_1}, r^+} \circ \cdots \circ \bar i^{H^{k-1}}_{\g_{b_k}, r^+}$. This is a well-defined map from $\overline \CH^{{H^\gamma}, \, \flat}_{b_k, r^+}$ to $\overline \CH^G_{b_1, r^+}$. In particular, we have $\bar i_{\g, r^+}: \overline \CH^{{H^\gamma}, \, \flat}_{r, r^+} \to \overline \CH^{G}_{b_1, r^+}$. 

For any $x \in \Bd(\gamma_{\le r})$ and $h \in H^\gamma_{x, r} \cap H^\gamma_{r^+}$. By definition, 
\begin{align*} 
\bar i_{\g_{b_1}, \g_{b_2}, \cdots, \g_{b_k}, r^+} & \left(\mathbbm{1}_{h H^\gamma_{x, r^+}}\right) \equiv \bar i_{\g_{b_1}, \g_{b_2}, \cdots, \g_{b_{k-1}}, r^+}\left(\frac{\mu_{{H^\gamma}}(H^\gamma_{x, r^+})}{\mu_{H^{k-1}}(H^{k-1}_{x, r^+})} \mathbbm{1}_{\g_{b_k} h H^{k-1}_{x, r^+}} \right) \\ 
& \equiv \bar i_{\g_{b_1}, \g_{b_2}, \cdots, \g_{b_{k-2}}, r^+} \left(\frac{\mu_{{H^\gamma}}(H^\gamma_{x, r^+})}{\mu_{H^{k-1}}(H^{k-1}_{x, r^+})} \frac{\mu_{H^{k-1}}(H^{k-1}_{x, r^+})}{\mu_{H^{k-2}}(H^{k-2}_{x, r^+})}  \mathbbm{1}_{\g_{b_{k-1}} \g_{b_k} h H^{k-2}_{x, r^+}} \right)  \\ 
& \equiv \cdots \\ & \equiv \frac{\mu_{{H^\gamma}}(H^\gamma_{x, r^+})}{\mu_{H^{k-1}}(H^{k-1}_{x, r^+})} \frac{\mu_{H^{k-1}}(H^{k-1}_{x, r^+})}{\mu_{H^{k-2}}(H^{k-2}_{x, r^+})} \cdots \frac{\mu_{H^1}(H^1_{x, r^+})}{\mu_{G}(G_{x, r^+})} \mathbbm{1}_{\g h G_{x, r^+}}  \\ & \equiv \frac{\mu_{{H^\gamma}}(H^\gamma_{x, r^+})}{\mu_G(G_{x, r^+})} \mathbbm{1}_{\g h G_{x, r^+}}\mod{[\CaH,\CaH]}.
\end{align*}
Thus the map $\bar i_{\g_{b_1}, \g_{b_2}, \cdots, \g_{b_k}, r^+}$ only depends on $\g$. This finishes the proof.
\end{proof}


\proof[Proof of Theorem \ref{JD-H}] 
We have shown in Corollary \ref{3.3.3} that $\overline\CaH_{r^+}^{\rig}=\bigoplus_{[\gamma]_\rg\in\CaS_r^\rg}\overline\CaH_{[\gamma]_\rg}^{G,\flat}$.

Let $\gamma_z\gamma_{b_1}$ be a $G$-good mod center element and $\chf_{gG_{y,r^+}}\in\CaH^{G,\flat}_{[\gamma_z\gamma_{b_1}]_\rg}$. By Proposition \ref{good-i}, $\bar i_{\g_{b_1}, r^+}: \overline \CH^{H^1, \, \flat}_{t, r^+} \to \overline \CH^{G}_{[\gamma_z\gamma_{b_1}]_\rg,r^+}$, where $H^1=C_G(\gamma_z\gamma_{b_1})$,  is a well defined surjective map. 
Write $g=\gamma_z\gamma_{b_1}u_1$ where $u_1\in H^1_{y,b_1}\cap H^1_{b_1^+}$.
Then, $\chf_{u_1H^1_{y,r^+}}\equiv \sum_{j}c_j\chf_{\gamma_{b_{2j}}u_{2j}H^1_{y_{2j},r^+}}$ for some $H^1$-good elements $\gamma_{b_{2j}}$ of depth $b_{2j}>b_1$, $y_{2j}\in\Bd(H^{2j})$ where $H^{2j}=C_{H^1}(\gamma_{b_{2j}})$, 
$c_j\in\bbZ[\frac1p]$ and $u_{2j}\in H^{2j}_{y_{2j}, b_{2j}}\cap H^{2j}_{b_{2j}^+}$. 
Note that there are constants $c'_j\in\bbZ[\frac1p]$ such that
\[
\sum_{j}c_j'\cdot\chf_{\gamma_{b_{2j}}u_{2j}H^1_{y_{2j},r^+}}\overset{\bar  i_{\g_{b_1}, r^+}}\longrightarrow
\sum_jc_j\cdot\chf_{\gamma_z\gamma_{b_1}\gamma_{b_{2j}}u_{2j}G_{y_{2j},r^+}}\equiv \chf_{gG_{y,r^+}}.
\]
Repeating the process to each summand, 
$\chf_{u_{2j}H^{2j}_{y,r^+}}\equiv \sum_k  c_k\cdot \chf_{\gamma_{b_{2jk}}u_{2jk}H^{2j}_{y_{2jk},r^+}}$ 
for some $H^{2j}$-good elements $\gamma_{b_{2jk}}$ of depth $b_{2jk}>b_{2j}$, $y_{2jk}\in\Bd(H^{2jk})$ 
where $H^{2jk}=C_{H^{2j}}(\gamma_{b_{2jk}})$ and 
$u_{2jk}\in H^{2jk}_{y_{2jk}, b_{2jk}}\cap H^{2jk}_{b_{2jk}^+}$. Now
\begin{align*}
&\sum_k \sum_jc_{jk}\cdot \chf_{\gamma_{b_{2jk}}u_{2jk}H^1_{y_{2jk},r^+}} \overset{\sum_j\bar  i_{\g_{b_{2j}}, r^+}}\longrightarrow 
\sum_{j}c_j'\chf_{\gamma_{b_{2j}}\gamma_{b_{2jk}}u_{2j}H^1_{y_{2j},r^+}}\\
&\overset{\bar  i_{\g_{b_1}, r^+}}\longrightarrow
\sum_jc_j\chf_{\gamma_z\gamma_{b_1}\gamma_{b_{2j}}\gamma_{b_{2jk}}u_{2j}G_{y_{2j},r^+}}\equiv \chf_{gG_{y,r^+}} \mod [\CH, \CH].
\end{align*}
Setting $b_i$ to be the min of depths appearing in summands in each $i$-th step. One can repeat the process until $b_i>r$. This is a finite step since $b_i$ forms an increasing discrete sequence.  Now, we proved that
$$\overline\CaH^\rig\subset\sum_{[\gamma] \in \CaS_r} \bar i_{\g, r^+} (\overline \CH^{C_G(\g), \, \flat}_{r, r^+})  \subset\sum_{[\gamma]\in\CaS_r}\overline\CaH_{[\gamma]}^{G,\flat} \subset \sum_{[\gamma]\in\CaS_r}\overline \CH_{r^+}({}^G\!(\gamma H^\gamma_{r^+})) \subset \overline\CH^\rig.$$ 
Therefore all the inclusions above are in fact equalities and for any $[\gamma]\in\CaS_r$, we have $$\bar i_{\g, r^+} (\overline \CH^{H^\gamma, \, \flat}_{r, r^+})=\overline\CaH_{[\gamma]}^{G,\flat}=\overline \CH_{r^+}({}^G\!(\gamma H^\gamma_{r^+})).$$

Also $\overline \CH^\rig=\sum_{[\gamma]\in\CaS_r}\overline \CH_{r^+}({}^G\!(\gamma H^\gamma_{r^+}))$. By Lemma \ref{domain}, this is a direct sum.
\qed

\smallskip

As shown in the proof, we have the following description of $\overline \CH_{r^+}({}^G\!(\gamma H^\gamma_{r^+}))$.
\begin{corollary}\label{cor: surj i}
Let $\bar i_{\g,r^+}$ be as in Theorem \ref{thm: bar i}. Then, $\bar i_{\g, r^+} (\overline \CH^{H^\gamma, \, \flat}_{r, r^+})=\overline\CaH_{[\gamma]}^{G,\flat}$.
\end{corollary}

\section{\bf Jordan Decomposition of $\overline\CaH_{r^+}^\rig$}\label{sec: JD}

\subsection{The cosets $I^u_r$ and $I^d_r$}\label{4.1} Following \cite{De1, De2}, we set 
\begin{gather*} I_r(G)=\{(F^*, X); F^* \text{ is a generalized $r$-facet of } G, X \in G_{F^*}/G_{F^*}^+\}, \\
I^u_r(G)=\{(F^*, X) \in I_r(G); X=u G_{F^*}^+ \text{ for some unipotent element } u \in G_{F^*}\}.
\end{gather*}

By \cite[Corollary 3.7.10]{AD}, $I^u_r(G)=\{(F^*, X) \in I_r(G); X \subset G_{r^+}\}$. 

By \cite[Definition 5.3.4]{De1}, under Hypotheses (DB), to each pair $(F^*, X) \in I^u_r(G)$, there exists a unique unipotent conjugacy class of minimal dimension which intersects $X$. We denote this unipotent conjugacy class by $\CO(F^*, X)$. 

Finally, we define the distinguished cosets $I^d_r(G) \subset I^u_r(G)$ as in \cite[Definition 5.5.1]{De1} and the equivalence relation $\sim$ as in \cite[Definition 3.6.2]{De1}. By \cite[Theorem 5.6.1]{De1}, under Hypotheses (DB), the map $(F^*, X) \mapsto \CO(F^*, X)$ gives a bijection between $I^d_r(G)/\sim$ and the set $Cl^u(G)$ of unipotent conjugacy classes of $G$. 

We first prove that

\begin{proposition}\label{JD-4'}
Suppose Hypotheses (DB) and Hypothesis \ref{conver} hold. Then $\overline \CH_{[1]}^{G,\flat}$ is a free $\BZ[\frac{1}{p}]$-module with basis $\mathbbm{1}_{(F^*, X)}$, where $(F^*, X)$ runs over representatives in $I^d_r(G)/\sim$. 
\end{proposition}

We adapt the strategy of \cite[\S 2]{De2}. While the invariant distributions (with complex coefficients) are considered in \cite{De2}, here we consider the cocenter of $\CH$ and need to work the coefficients in $\BZ\left[\frac{1}{p}\right]$. 
 


\begin{lemma}\label{O-O'}
Let  $\CO$ be a unipotent conjugacy class of $G$ and $(F_1^*, X_1), (F_2^*, X_2) \in I^u_r(G)$ are two pairs associated to $\CO$ (i.e., $\CO=\CO(F_1^*, X_1)=\CO(F_2^*, X_2)$) such that $F_2^* \subset \overline{F_1^*}$ and $X_2 \subset X_1$. Then in $\overline \CH$, we have $$\mathbbm{1}_{X_1} \in p^n \mathbbm{1}_{X_2}+\sum_{(F^*, X) \in I^u_r(G) \text{ with } \CO(F^*, X)>\CO} \BZ\left[\frac{1}{p}\right] \mathbbm{1}_X+[\CH, \CH] \quad \text{ for some } n \in \BN.$$
\end{lemma}

\begin{proof}
We follow the argument in \cite[Lemma 2.6.2]{De2}, almost verbatim. We have that $$G_{F_2^*}^+ \subset G_{F_1^*}^+ \subset G_{F_1^*} \subset G_{F_2^*}.$$ We write $\mathbbm{1}_{X_1}$ as 
$$\mathbbm{1}_{X_1}=\sum_{Y \in X_1/G_{F_2^*}^+} \mathbbm{1}_Y.$$

Let $Y \in X_1/G_{F_2^*}^+$. Then $Y \subset X_1 \subset G_{r^+}$. By \cite[Corollary 5.2.5]{De1}, we have $\CO(F_2^*, Y) \ge \CO$. The case where $\CO(F_2^*, Y)>\CO$ is obvious. It remains to consider the case where $\CO(F_2^*, Y)=\CO$. By \cite[Lemma 3.2.17]{De1}, there exists $x \in F_1^*$ so that $G_x \subset \stab_G(F_2^*)$. By \cite[Corollary 5.2.3]{De1}, we have $Y={}^g X_2$ for some $g \in G_x^+$, where $x \in F^*$.  In particular, $\mathbbm{1}_Y \equiv \mathbbm{1}_{X_2} \mod [\CH, \CH]$. Set $\Gamma=\{Y \in X_1/G_{F_2^*}^+; \CO(F_2^*, Y)=\CO\}$. Then $G_x^+$ acts transitively on $\Gamma$. Since $G_x^+$ is a pro-$p$ group, the cardinality of $\Gamma$ is a power of $p$. The statement is proved. 
\end{proof}

The following results easily from Lemma \ref{O-O'} and the definition of $\sim$ (see the proof of \cite[Lemma 2.6.5]{De2}). 

\begin{corollary}\label{equi-class}
Let $(F_1^*, X_1), (F_2^*, X_2) \in I^d_r(G)$ with $(F_1^*, X_1) \sim (F_2^*, X_2)$. Then $$\mathbbm{1}_{X_1} \in p^n \mathbbm{1}_{X_2}+\sum_{(F^*, X) \in I^u_r(G) \text{ with } \CO(F^*, X)>\CO} \BZ\left[\frac{1}{p}\right] \mathbbm{1}_X+[\CH, \CH] \quad \text{ for some } n \in \BZ.$$
\end{corollary}

\proof[Proof of Proposition \ref{JD-4'}] Note that $\overline \CH_{[1]}^{G,\flat}$ is spanned by $\mathbbm{1}_{(F^*, X)}$ for $(F^*, X) \in I^u_r(G)$. By definition, for any $(F^*, X) \in I^u_r(G)$, there exists $(F_1^*, X_1) \in I^d_r(G)$ such that $F^* \subset \overline{F_1^*}$ and $X \subset X_1$. By Lemma \ref{O-O'}, given a unipotent conjugacy class $\CO$, for any $(F^*, X) \in I^u_r(G)$ with $\CO(F^*, X)=\CO$, the element $\mathbbm{1}_{(F^*, X)}$ in $\overline \CH$ is spanned by $\mathbbm{1}_{(F_1^*, X_1)}$, where $(F_1^*, X_1) \in I^d_r(G)$ with $\CO(F_1^*, X_1)=\CO$ and $(F^{\prime\ast}, X') \in I^u_r(G)$ with $\CO(F^{\prime\ast}, X')>\CO$. Here, we denote $\chf_X$ by $\chf_{(F^\ast,X)}$ for clarity. 

By Corollary \ref{equi-class}, it suffices to use any representative $(F_1^*, X_1) \in I^d_r(G)/\sim$ with $\CO(F_1^*, X_1)=\CO$ instead of all the distinguished cosets associated to $\CO$. 

By Hypotheses (DB), there are only finitely many unipotent conjugacy classes. Hence by induction, $\mathbbm{1}_{(F^*, X)}$ is spanned by $\mathbbm{1}_{(F^{\prime\ast}, X')}$, where $(F^{\prime\ast}, X')$ runs over representatives in $I^d_r(G)/\sim$ with $\CO(F^{\prime\ast}, X') \ge \CO$. 
In particular, $\overline \CH_{[1]}^{G,\flat}$ is spanned by $\mathbbm{1}_{(F^{\prime\ast}, X')}$, where $(F^{\prime\ast}, X')$ runs over representatives in $I^d_r(G)/\sim$.

\smallskip

Now we choose a set of representatives $(F^{\prime\ast}, X')$ of $I^d_r(G)/\sim$. It remains to show that the elements $\mathbbm{1}_{(F^{\prime\ast}, X')}$ are linearly independent over $\mathbb Z[\frac{1}{p}]$. 
Suppose that $\sum a_{(F^{\prime\ast}, X')} \mathbbm{1}_{(F^{\prime\ast}, X')} \equiv 0$ in $\overline \CH$ with $a_{(F^{\prime\ast}, X')} \in \mathbb Z[\frac{1}{p}] \subset \mathbb C$. We regard $\sum a_{(F^{\prime\ast}, X')} \mathbbm{1}_{(F^{\prime\ast}, X')}$ as the zero element in $\overline \CH_{\mathbb C}$. Suppose that not all the coefficients $a_{(F^{\prime\ast}, X')}$ are $0$. Let $\CO$ be a minimal unipotent conjugacy class such that $\CO=\CO(F_1^{\prime\ast}, X'_1)$ for some $(F_1^{\prime\ast}, X'_1)$ with $a_{(F_1^{\prime\ast}, X'_1)} \neq 0$. By the minimality assumption, for any other representative $(F^{\prime\ast}, X')$ in our chosen set, we have $a_{(F^{\prime\ast}, X')}=0$ or $\CO \cap X'=\emptyset$. For $u \in \CO$, we have
\begin{align*} 
0 &=O_u (a_{(F_1^{\prime\ast}, X'_1)} \mathbbm{1}_{(F_1^{\prime\ast}, X'_1)})+ 
\sum_{(F^{\prime\ast}, X') \neq (F_1^{\prime\ast}, X'_1)} O_u(a_{(F^{\prime\ast}, X')} \mathbbm{1}_{(F^{\prime\ast}, X')}) \\ &=a_{(F_1^{\prime\ast}, X'_1)} O_u(\mathbbm{1}_{(F_1^{\prime\ast}, X'_1)}). 
\end{align*}

This is a contradiction as $a_{(F_1^{\prime\ast}, X'_1)} \neq 0$ and $O_u (\mathbbm{1}_{(F_1^{\prime\ast}, X'_1)})$ is a nonzero number in $\mathbb C$. Hence the image of $\mathbbm{1}_{(F^{\prime\ast}, X')}$ in $\overline \CH$ for any set of representatives $(F^{\prime\ast}, X')$ of $I^d_r(G)/\sim$ are linearly independent. 
\qed

\subsection{Unipotent orbits in the group $H^{\gamma}$} Now we consider $\overline \CH_{[\g]}^{G,\flat}$ for arbitrary $[\gamma] \in \CaS_r$. Set $\bH=\bH^{\gamma}$. Note that $\bH$ is not connected in general. The cosets we consider in this situation are the distinguished cosets $I^d_r(H):=I^d_r(H^\circ)$, but there are extra equivalence relations that we need to take into account. The equivalence relation $\tilde \sim$ on $I^d_r(H)$ is generated by the equivalence relation $\sim$ on $I^d(H^\circ)$ in \cite[Definition 3.6.2]{De1} and the relation $(F^*, X) \tilde \sim ({}^h F^*, {}^h X)$ for $h \in H$. In other words, the group $H/H^\circ$ acts naturally on $I^d_r(H^\circ)/\sim$ and the quotient set is $I^d_r(H)/\tilde \sim$. 

On the other hand, let $Cl^u(H)$ be the set of unipotent conjugacy classes of $H$ and $Cl^u(H^\circ)$ be the set of unipotent conjugacy classes of $H^\circ$. Under Hypothesis \ref{hyp: unip}, the natural map $Cl^u(H^\circ) \to Cl^u(H)$ is surjective. The group $H/H^\circ$ acts naturally on $Cl^u(H^\circ)$ and the quotient set is $Cl^u(H)$.

It is easy to see that the map $I^d_r(H^\circ)/\sim \to Cl^u(H^\circ)$ given by $(F^*, X) \mapsto \CO(F^*, X)$ is $H/H^\circ$-equivariant. Thus it leads to a map $I^d_r(H)/\tilde \sim \to Cl^u(H)$. Combining this with the result in \cite[\S4.4]{De2}, under Hypotheses (DB) and Hypothesis \ref{hyp: unip}, this map is bijective. 

Now we come to the main result of this section. 

\begin{theorem}\label{JD-4}
Suppose Hypotheses (DB), Hypothesis \ref{hyp: unip} and Hypothesis \ref{conver} hold. Then for any $[\gamma] \in \CaS_r$,

(1) $\overline \CH^{H^\gamma, \flat}_{r,r^+}$ is a free $\BZ[\frac{1}{p}]$-module with basis $\mathbbm{1}_{(F^*, X)}$, where $(F^*, X)$ runs over representatives in $I^d_r(H^\gamma)/\tilde \sim$. 

(2) The map $\overline i_{\g, r^+}: \overline \CH^{H^\gamma, \flat}_{r,r^+} \to \overline \CH_{[\g]}^{G,\flat}$ defined in Theorem \ref{thm: bar i} is a $\bbZ[\frac1p]$-linear isomorphism. 
\end{theorem}

\begin{proof} Without loss of generality, we assume $\gamma=\gamma_{\le r}$.

By definition, if $h \in H^\gamma$, then for any $(F^*, X) \in I^d_r(H^\gamma)$, we have $\mathbbm{1}_{(F^*, X)} \equiv \mathbbm{1}_{({}^h F^*, {}^h X)} \mod [\CH^{H^\gamma}, \CH^{H^\gamma}]$. We choose a set of representatives $(F^*, X)$ in  in $I^d_r(H^\gamma)/\tilde \sim$. Similar to the argument of Proposition \ref{JD-4'}, $\overline \CH^{H^\gamma}_{r,r^+}$ is spanned by $\mathbbm{1}_{(F^*, X)}$. By Corollary \ref{cor: surj i}, $\bar i_{\g, r^+} (\overline \CH^{H^\gamma, \, \flat}_{r, r^+})=\overline\CaH_{[\gamma]}^{G,\flat}$.

Now suppose that $\overline i_{\g, r^+} (\sum a_{(F^*, X)} \mathbbm{1}_{(F^*, X)}) \equiv 0$ in $\overline \CH$, here $a_{(F^*, X)} \in \mathbb Z[\frac{1}{p}] \subset \mathbb C$. We regard $\overline i_{\g, r^+} (\sum a_{(F^*, X)} \mathbbm{1}_{(F^*, X)})$ as the zero element in $\overline \CH_{\mathbb C}$. Suppose that not all the coefficients $a_{(F^*, X)}$ are $0$. Let $\CO$ be a minimal unipotent conjugacy of $H^\gamma $ such that there exists $(F^*, X)$ in the chosen set of representatives with $\CO(F^*, X)=\CO$ and $a_{(F^*, X)} \neq 0$. 

By Lemma \ref{lem: good element centralizer}, the map $g \mapsto \g g$ induces an injective map from the set of conjugacy classes of $H^{\gamma}$ to the set of conjugacy classes of $G$. Let $\CO'$ be the conjugacy class of $G$ that contains $\g \CO$. Note that the support of $\overline i_{\g, r^+}(\mathbbm{1}_{(F^\ast,X)})$ is contained in ${}^G X$. Let $u \in \CO$. For any element $(F^{\prime\ast}, X')$ in our chosen set, if $(F^{\prime\ast}, X') \neq (F^*, X)$, then $O_{\gamma u} (\overline i_{\g, r^+} (a_{(F^{\prime\ast}, X')} \mathbbm{1}_{(F^{\prime\ast}, X')}) )=0$. Then \begin{align*} 0 &=O_{\gamma u}(\overline i_{\g, r^+}(a_{(F^*, X)} \mathbbm{1}_{(F^*, X)}))+ \sum_{(F^{\prime\ast}, X') \neq (F^*, X)} O_{\gamma u}(\overline i_{\g, r^+}(a_{(F^{\prime\ast}, X')} \mathbbm{1}_{(F^{\prime\ast}, X')})) \\ &=a_{(F^*, X)} O_{\gamma u}(\overline i_{\g, r^+}(\mathbbm{1}_{(F^*, X)})). \end{align*}

This is a contradiction as $a_{(F^*, X)} \neq 0$ and $O_{\gamma u}(\overline i_{\g, r^+}(\mathbbm{1}_{(F^*, X)}))$ is a nonzero number in $\mathbb C$. Therefore the set $\mathbbm{1}_{(F^*, X)}$ is linearly independent and the map $\overline i_{\g, r^+}$ is injective. 
\end{proof}

\begin{corollary}
Suppose Hypotheses in \S\ref{subsec: hypos} hold. Then $\overline \CH^\rig_{r^+}$ is a free $\BZ[\frac{1}{p}]$-module. If moreover $G$ is semisimple, then the rank of $\overline \CH^\rig_{r^+}$ is $\sum_{[\gamma] \in \CaS_r} \sharp Cl^u(H^{\gamma})$.
\end{corollary}

\subsection{Application to invariant distributions} For any compact subset $X$ of $G$, we denote by $J(X)$ the space of complex-valued invariant distributions of $G$ with support in ${}^G X$. Similarly, we write $\overline\CH(X)$, $\overline\CH(X)_{r^+,\mathbb C}$ for $\overline\CH({}^G\!X)$, $\overline\CH({}^G\!X)_{r^+,\mathbb C}$ etc. for simplicity. Now we discuss some application to the invariant distributions. We first recall Theorem \ref{thmB} in the introduction and give a proof of it. 

\begin{theorem}\label{thmB'}
Suppose Hypotheses in \S\ref{subsec: hypos} hold. The restriction $J(G^\rig) \mid_{\CH_{r^+, \mathbb C}}$ has a basis given by the restriction of orbital integrals $O_{\g_{\le r} u}$ to ${\CH_{r^+, \mathbb C}}$, where $[\g] \in \CS_r$, and $u$ runs over the representatives of the unipotent conjugacy classes of $H^{\gamma}$. 
\end{theorem}

\proof 
By Theorem \ref{JD-H}, $\overline \CH^\rig_{r^+, \mathbb C}=\oplus_{[\gamma] \in \CaS_r} \overline \CH_{r^+, \mathbb C}(\gamma H^\gamma_{r^+})$ and each subset ${}^G\!(\gamma H^\gamma_{r^+})$ is a $G$-domain. We have \begin{align*} J(G^\rig) \mid_{\CH_{r^+, \mathbb C}} &=J(G^\rig) \mid_{\overline \CH_{r^+, \mathbb C}}=J(G^\rig) \mid_{\overline \CH^\rig_{r^+, \mathbb C}}=\oplus_{[\gamma] \in \CaS_r} J(G^\rig) \mid_{\overline \CH_{r^+, \mathbb C}(\gamma H^\gamma_{r^+})} \\ &=\oplus_{[\gamma] \in \CaS_r} \overline \CH_{r^+, \mathbb C}(\gamma H^\gamma_{r^+})^*.\end{align*} 

For any $[\gamma], [\gamma'] \in \CaS_r$ with $[\gamma'] \neq [\gamma]$, by Proposition \ref{prop: dec Gc} we have ${}^G\!(\gamma H^{\gamma}_{r^+}) \cap {}^G\!(\gamma' H^{\gamma'}_{r^+})=\emptyset$ and hence for any unipotent element $u \in H^\gamma$, $O_{\gamma_{\le r} u} (\overline \CH_{r^+, \mathbb C}(\gamma' H^{\gamma'}_{r^+}))=0$. 

Now we fix an equivalence class $[\gamma] \in \CaS_r$ and an element $\gamma_{\le r}$. By the proof of Theorem \ref{JD-4} (2), the dimension of $\overline \CH_{r^+, \mathbb C}(\gamma H^\gamma_{r^+})$ equals to the number of unipotent conjugacy classes of $H^{\gamma}$ and the orbital integrals $O_{\gamma_{\le r} u}$, where $u$ runs over representatives of the unipotent conjugacy classes of $H^{\gamma}$, form a basis of linear functions on $\overline \CH_{r^+, \mathbb C}(\gamma H^\gamma_{r^+})$. The theorem is proved.
\qed

\subsection{}\label{howe-diss} Finally, we explain how Theorem \ref{thmB'} may be applied to Howe's conjecture. 
In \cite{Howe}, Howe conjectured that for any open compact subgroup $K$ and compact subset $X$ of $G$, the restriction $J(X) \mid_{\CH_{\mathbb C}(G, K)}$ is finite dimensional. This is proved by Clozel \cite{Cl} and by Barbasch and Moy \cite{BM}. Another proof is given by the first-named author in \cite{hecke-1}. 

Following \cite{hecke-1}, we have the Newton decompositions 
$$
G=\sqcup_{\nu \in \aleph} G(\nu) \quad \text{ and } \quad G^\rig=\sqcup_{\nu \in \aleph;\\ \,C_G(\nu)=G }\  G(\nu),
$$ 
where $\aleph$ is the product of $\pi_1(G)$ (the Kottwitz factor) and the set of dominant rational coweights of $G$ (the Newton factor), and $G(\nu)$ is the corresponding Newton stratum defined in \cite[\S 2.2]{hecke-1}.

It follows from the definitions of Newton strata and $r^+$-equivalence that for semisimple compact-modulo-center elements, if $\g$ and $\g'$ are $r^+$-equivalent mod center (for some $r$), then $\g$ and $\g'$ are contained in the same Newton stratum. For $\nu \in \aleph$ that is central in $G$, we let $\CS_{\nu, r}$ be the set of $r^+$-equivalence classes of semisimple elements in $G(\nu)$. Then we have 
$$
\CS_r=\sqcup_{\nu \in \aleph;\, \,C_G(\nu)=G} \ \CS_{\nu, r}.
$$ 

Based on the approach of \cite{hecke-1}, the study of the restriction $J(X) \mid_{\CH_{\mathbb C}(G, K)}$ can be reduced to the study of $J(G(\nu)) \mid_{\CH_{\mathbb C}(G, I_n)}$, where $\nu \in \aleph$ that is central in $G$ and $I_n$ is the $n$-th congruent subgroup of an Iwahori subgroup of $G$.  If $r=n-\e$, where $\e$ is a sufficiently small positive number, then there is an $x$ in the base alcove such that $G_{x, r^+}=I_n$ and $\overline \CH_{r^+}=\overline \CH(G, I_n)$. In this case, $J(G^\rig) \mid_{\CH_{r^+, \mathbb C}}=J(G^\rig) \mid_{\CH_{\mathbb C}(G, I_n)}$. 

Let $\CH(G, I_n; \nu)$ be the $\mathbb Z[\frac{1}{p}]$-submodule of $\CH(G, I_n)$ consisting of functions with support in $G(\nu)$ and $\overline \CH(G, I_n; \nu)$ be the image of $\CH(G, I_n; \nu)$ in the cocenter $\overline \CH$. The main result of \cite{hecke-1} establishes the Newton decomposition (see \cite[Theorem 4.1]{hecke-1}): 
$$
\overline \CH(G, I_n)=\sqcup_{\nu \in \aleph}\  \overline \CH(G, I_n; \nu) \quad \text{ and } 
\quad \overline \CH(G, I_n)^\rig=\sqcup_{\nu \in \aleph; \,C_G(\nu)=G} \ \overline \CH(G, I_n; \nu).$$

Combining it with Theorem \ref{thmB'}, we have

\begin{theorem}
Suppose Hypotheses in \S\ref{subsec: hypos} hold. Let $\nu \in \aleph$ such that $\nu$ is central in $G$. The restriction $J(G(\nu)) \mid_{\CH_{\mathbb C}(G, I_n)}$ has a basis given by the restriction of orbital integrals $O_{\g_{\le r} u}$ to ${\CH_{\mathbb C}(G, I_n)}$, where $[\g] \in \CS_{\nu, r}$, and $u$ runs over the representatives of the unipotent conjugacy classes of $H^{\gamma}$. 

In particular, the dimension of  $J(G(\nu)) \mid_{\CH_{\mathbb C}(G, I_n)}$ is equal to $\sum_{[\gamma] \in \CaS_{\nu, r}} \sharp Cl^u(H^{\gamma})$.
\end{theorem}

This result gives an explicit basis of the finite dimensional space $J(G(\nu)) \mid_{\CH_{\mathbb C}(G, I_n)}$, and thus gives a precise estimate on the finiteness of the restriction of invariant distributions predicted by Howe in \cite{Howe}.

\section{Examples}\label{sec: ex}

In this section, we give some examples to illustrate relations between the cocenter and the representations. We will work with the Hecke algebras of complex-valued functions and complex representations. 

\subsection{Cocenter and representations} Before we come to some concrete examples, we would like to give a brief discussion on the relation between the cocenters and the representations. 

Recall that $\mathfrak R_{\mathbb C}(G)$ is the complexified Grothendieck group of smooth admissible complex representations of $G$ of finite length. Let $\mathcal P$ be the set of all proper parabolic subgroups of $G$. For any Levi subgroup $M$ of $G$, we denote by $\Psi(M)_{\mathbb C}$ the group of unramified character of $M$ over $\mathbb C$. We define the {\it elliptic quotient} and the {\it rigid quotient} as follows:
\begin{gather*}
\mathfrak R_{\mathbb C}(G)_{ell}=\mathfrak R_{\mathbb C}(G)/\langle \textrm{Ind}_P^G(\sigma))\mid P=MN \in \mathcal P, \sigma \in \mathfrak R_{\mathbb C}(M) \rangle;\\
\mathfrak R_{\mathbb C}(G)_{\rig}=\mathfrak R_{\mathbb C}(G)/\langle \textrm{Ind}_P^G(\sigma)-\textrm{Ind}_P^G(\sigma\otimes\chi)\mid P=MN \in \mathcal P, \sigma \in \mathfrak R_{\mathbb C}(M), \chi \in \Psi(M)_{\mathbb C} \rangle.
\end{gather*}

We have discussed the rigid cocenter $\overline \CH^\rig_{\mathbb C}$ in this paper. There is another important subspace of cocenter, the elliptic cocenter $\overline \CH^{ell}_{\mathbb C}$, introduced by Bernstein, Deligne and Kazhdan in \cite{BDK}. By definition, $$\overline \CH^{ell}_{\mathbb C}=\{f \in \overline \CH_{\mathbb C}; \bar r_M(f)=0 \text{ for any proper standard Levi } M\},$$ where $\bar r_M: \overline \CH_{\mathbb C}^G \to \overline \CH_{\mathbb C}^M$ is the map adjunct to the parabolic induction functor $\mathfrak R_{\mathbb C}(M) \to \mathfrak R_{\mathbb C}(G)$. 

The trace map $Tr_{\mathbb C}: \overline \CH_{\mathbb C} \to \mathfrak R_{\mathbb C}(G)^*$ induces 
\[\tag{a}
Tr_{\mathbb C}: \overline \CH^{ell}_{\mathbb C} \to \mathfrak R_{\mathbb C}(G)_{ell}^*, \qquad Tr_{\mathbb C}: \overline \CH^{\rig}_{\mathbb C} \to \mathfrak R_{\mathbb C}(G)_{\rig}^*.
\]
Here the first map is studied in \cite{BDK} and the second map is studied in \cite{hecke-3}. 

Similarly, for any $n \in \BN$, let $\mathfrak R(\CH_{\mathbb C}(G, I_n))$ be the complexified Grothendieck group of finite dimensional representations of $\CH_{\mathbb C}(G, I_n)$. We denote by $\mathfrak R(\CH_{\mathbb C}(G, I_n))_{ell}$ and $\mathfrak R(\CH_{\mathbb C}(G, I_n))_{\rig}$ the elliptic quotient and the rigid quotient of $\mathfrak R(\CH_{\mathbb C}(G, I_n))$ respectively. Then we have 
\[\tag{b} 
Tr_{\mathbb C}: \overline \CH_{\mathbb C}(G, I_n)^{ell} \to \mathfrak R(\CH_{\mathbb C}(G, I_n))_{ell}^*, \quad Tr_{\mathbb C}: \overline \CH_{\mathbb C}(G, I_n)^{\rig} \to \mathfrak R(\CH_{\mathbb C}(G, I_n))_{\rig}^*.
\]

If $G$ is semisimple, then all the vector spaces in (b) are finite dimensional and the maps in (b) are bijective. Here the surjectivity follows from the trace Paley-Wiener theorem \cite{BDK} and \cite{hecke-3} and the injectivity follows from the density theorem \cite{Kaz}.

\subsection{The $PGL_2(F)$ case} In this subsection, we assume that $G=PGL_2(F)$, where $F$ is a locally compact field with finite residue field $\mathbb F_q$. We assume furthermore that $q$ is odd.

Up to conjugation, there is 
\begin{itemize}
\item a unique split maximal torus of $G$, which we denote by $T_s$;
\item a unique maximal torus that split over the unramified extension of $F$, which we denote by $T_u$;
\item two non-conjugate maximal tori that split over ramified extensions, which we denote by $T_{rm}$ and $T_{rm}'$. 
\end{itemize}

Let $n$ be a positive integer and $r=n-\e$, where $\e$ is sufficiently small. For $[H]\in\{G, T_u,T_{rm},T_{rm}^{\prime}\}$ and $[\gamma]\in\CaS_r$, $\sharp_{[H]}$ denotes the cardinality of $\{[\gamma]\in\CaS_r\mid [C_G(\gamma_{\le r})]=[H]\}$. Then, from the table below, we have
$$\dim \overline \CH_{\mathbb C}(G, I_n)^{\rig}=3 q^n+2.$$

\[
\begin{tabular}{|c|c|c|}
\hline
$[H]$& $\sharp_{[H]}$ & $Cl_{H}^u$\\
\hline\hline
$PGL_2$& 1&2 \\
\hline
$T_s$ & $\frac{q^{n-1}(q-1)}{2}$ &1 \\
\hline
$T_u$ & $\frac{q^{n-1}(q+1)}2$ & 1\\
\hline
$T_{rm}$ & $q^n$ & 1\\
\hline
$T_{rm}^\prime$ & $q^n$ & 1\\
\hline
\end{tabular}
\]

\

Note that every $r^+$-equivalence class in $G^\rig$ contains some elliptic semisimple elements. However, in general, not every $r^+$-equivalence class in $G^\rig$ consists only of elliptic semisimple elements. One may show that for $G=PGL_2$, an $r^+$-equivalence class $[\g]$ in $G^\rig$ consists only of elliptic semisimple elements if and only if $H^{\g}$ is a compact subgroup of $G$. We define $\CaS^{ell}_r \subset \CaS_r$ to be the subset of $r^+$-equivalence classes in $G^\rig$ only consisting of elliptic semisimple elements. Then we have the following identity 
$$
\dim \overline \CH_{\mathbb C}(G, I_n)^{ell}=\dim \overline \CH_{\mathbb C}(G, I_n)^\rig-\dim \overline \CH_{\mathbb C}(T_s^c, T_{s, n})=\sharp \CaS^{ell}_r+1=2q^n+\frac{q^{n-1}(q+1)}{2}+1,
$$
where $T_s^c \subset T_s$ is the subgroup consisting of compact elements in $T_s$, $T_{s, n}$ is the $n$-th congruent subgroup of $T_s$ and the number $1$ in the third term comes from the regular unipotent conjugacy class of $G$. 

Note that the discrete series gives a natural basis of $\mathcal R_{\mathbb C}(G)_{ell}$ and the discrete series of depth at most $r$ gives a natural basis $\mathcal R_{\mathbb C}(\CH(G, I_n))_{ell}$. Moreover, the discrete series consist of supercuspidal representations and four non-supercuspidal discrete series representations. By direct calculation, one can check the number of supercuspidal representations of depth at most $r$ equals $\sharp \CaS^{ell}_r-3$, and we have $\dim \overline \CH_{\mathbb C}(G, I_n)^{ell}=\dim \mathfrak R(\CH_{\mathbb C}(G, I_n))_{ell}$. 

\subsection{Quaternion algebra} Let $G=PGL_2(F)$. Let $D$ be a quaternion algebra over $F$ and $G'=D^\times/F^\times$. It is well-known that there is a natural bijection between the set of elliptic semisimple conjugacy classes in $GL_2(F)$ and the regular semisimple conjugacy classes in $D^\times$. Here $\gamma \leftrightarrow \gamma'$ if and only if they have the same characteristic polynomial. Therefore, there is a natural bijection between the set of elliptic semisimple conjugacy classes in $G=PGL_2(F)$ and the regular semisimple conjugacy classes in $G'$ and this bijection preserves the depth. We have that $$\dim \overline \CH_{\mathbb C}(G', I_n^{G'})=\sharp \CaS^{G'}_r=\sharp \CaS^{G, ell}_{r}+1=\dim \overline \CH_{\mathbb C}(G, I_n)^{ell}.$$ Here the number $1$ in the third term comes from the $r^+$ equivalence $[1]$ in $\CaS^{G'}_r$.

The local Jacquet-Langlands correspondence \cite{JL} gives a bijection between the discrete series of $G$ and the irreducible representations of $G'$. The natural duality between the cocenter and representations indicates that there is not only the numerical identity $\dim \overline \CH_{\mathbb C}(G', I_n^{G'})=\dim \overline \CH_{\mathbb C}(G, I_n)^{ell}$, but there also should be a natural bijection between the cocenter $\overline \CH_{\mathbb C}(G', I_n^{G'})$ and the elliptic cocenter $\overline \CH_{\mathbb C}(G, I_n)^{ell}$. It would be interesting to study such natural bijections for the (elliptic) cocenters of $PGL_m$ and its inner forms for arbitrary $m$.

\subsection{The $SL_2$ case} For $G=SL_2$, there are two non-conjugate maximal tori that split over the unramified extensions, which we denote by $T_u$ and $T_u^\prime$. We have the following table for $SL_2$ (with $n=1$)
\[
\begin{tabular}{|c|c|c|}
\hline
$[H]$& $\sharp_{[H]}$ & $Cl_{H}^u$\\
\hline\hline
$SL_2$& 2&5 \\
\hline
$T_s$ & $\frac{q-3}2$ & 1\\
\hline
$T_u$ & $\frac{q-1}2$ & 1\\
\hline
$T_u^\prime$ & $\frac{q-1}2$ & 1\\
\hline
$T_{rm}$ & $q-1$ & 1\\
\hline
$T_{rm}^\prime$ & $q-1$ & 1\\
\hline
\end{tabular}
\]

We have
$$\dim \overline \CH_{\mathbb C}(G, I_n)^{\rig}=3q+6+\frac{q-1}2, \quad \dim \overline \CH_{\mathbb C}(G, I_n)^{ell}=3 q+5.$$

We have seen that for $G=PGL_2$, $\dim \overline \CH_{\mathbb C}(G, I_1)^{\rig}=3 q+2$  and $\dim \overline \CH_{\mathbb C}(G, I_n)^{ell}=2q+\frac{q+3}{2}$. Thus the elliptic/rigid cocenters for $PGL_2$ and $SL_2$ are different.


\begin{thebibliography}{99}

\bibitem{Ad}
J.~Adler,
\emph{Refined anisotropic $K$-types and supercuspidal representations},
Pacific J. Math., \textbf{185} (1998), 1--32.

\bibitem{AD}
J. D. Adler and S. DeBacker, \emph{Some applications of Bruhat--Tits theory to harmonic analysis on the Lie algebra of a reductive $p$-adic group}, Michigan Math. J. \textbf{50} (2002), 263--286. 


\bibitem{AS}
J. D. Adler and L. Spice, \emph{Good product expansions for tame elements of p-adic groups},
Int. Math. Res. Pap. IMRP (2008), no. 1, Art. ID rp. 003, 95 pp.

\bibitem{BM}
D. Barbasch and A. Moy, \emph{A new proof of the Howe conjecture}, J. Amer. Math. Soc. 13 (2000), no. 3, 639--650. 

\bibitem{BDK}
J.~Bernstein, P.~Deligne, D.~Kazhdan, \emph{Trace Paley-Wiener theorem for reductive $p$-adic groups}, J. d'Analyse Math. {\bf 47} (1986), 180--192.

\bibitem{hecke-3}
D. Ciubotaru and X. He, \emph{Cocenters of $p$-adic groups, III: Elliptic cocenter and rigid cocenter}, arXiv:1703.00378.

\bibitem{Cl}
L. Clozel, \emph{Orbital integrals on p-adic groups: a proof of the Howe conjecture}, Ann. of Math. (2) 129 (1989), no. 2, 237--251.

\bibitem{De1}
S. DeBacker, \emph{Parametrizing nilpotent orbits via Bruhat-Tits theory}, Ann. of Math. (2) \textbf{156} (2002), 295--332.

\bibitem{De2}
S. DeBacker, \emph{Homogeneity results for invariant distributions of a reductive $p$-adic group}, Ann. Sci. \'Ecole Norm. Sup. (4), \textbf{35} (2002), 391--422.


\bibitem{Fintzen}
J. Fintzen, \emph{Tame tori in $p$-adic groups and good semisimple elements}, preprint

\bibitem{Hales}
T. C. Hales, \emph{Simple definition of transfer factors for unramified groups}, Representation Theory of Groups and Algebras, Contemp. Math. \textbf{145}, AMS, Providence, RI 1993, 109-134.

\bibitem{hecke-1}
X. He, \emph{Cocenters of $p$-adic groups, I: Newton decomposition}, 1610.04791.

\bibitem{hecke-2}
X. He, \emph{Cocenters of $p$-adic groups, II: induction map}, arXiv:1611.06825.

\bibitem{Howe}
R. Howe, \emph{Two conjectures about reductive p-adic groups}, Proc. of A.M.S. Symposia in Pure Math. XXVI (1973), 377--380.

\bibitem{JL}
H. Jacquet and R. P. Langlands, \emph{Automorphic Forms on $GL(2)$}, Lecture Notes in Math., vol. 114, Springer-Verlag, Berlin, 1970.

\bibitem{Kaz}
D.~Kazhdan, \emph{Cuspidal geometry of $p$-adic groups}, J. Analyse Math. {\bf 47} (1986), 1--36.

\bibitem{Kim} J.-L. Kim, \emph{Supercuspidal representations: An Exhaustion Theorem}, Journal of AMS, {\bf 20} (2007), no. 2, pp 1199-1234. 

\bibitem{KMu}
J.-L. Kim and F. Murnaghan, \emph{Character expansions and unrefined minimal K-types},
Amer. J. Math. \textbf{125} (2003), no. 6, 1199--1234.

\bibitem{KMu2}
J.-L. Kim and F. Murnaghan, \emph{$K$-types and $\Gamma$-asymptotic expansions}, J. Reine Angew. Math. {\bf 592} (2006), 189-236. 

\bibitem{Mc}
G. J. McNinch, \emph{Nilpotent orbits over ground fields of good characteristic}, Math. Ann. \textbf{329} (2004), 49--85.

\bibitem{MP}
A.~Moy and G.~Prasad, 
\emph{Unrefined minimal ${K}$-types for $p$-adic groups},
Inv. Math. \textbf{116} (1994), 393--408.

\bibitem{MP2}
A.~Moy and G.~Prasad, 
\emph{Jacquet functors and unrefined minimal ${K}$-types},
Comment. Math. Helvetici \textbf{71} (1996), 98--121.

\bibitem{Pr}
G. Prasad,
\emph{Galois fixed points in the Bruhat-Tits building of 
a reductive group}, 
Bull. Soc. Math. France, \textbf{129} (2001), 169--174.

\bibitem{Rao}
R. Ranga Rao, \emph{Orbital integrals in reductive groups}, Annals of Math. \textbf{96} (1972), 505--510.

\bibitem{Rous} 
G. Rousseau, 
\emph{Immeubles des groupes r\'eductifs sur les corps locaux},
Th\`ese, Paris XI (1977).

\bibitem{Spice}
L. Spice,
\emph{Topological Jordan decompositions}, Journal of Algebra, \textbf{319} (2008), 3141-3163.

\end{thebibliography}
\end{document}